\pdfoutput=1
\documentclass[10pt]{article}

\usepackage{amsmath}
\usepackage{amsthm}
\usepackage{graphicx}
\usepackage{color}

\usepackage{amsfonts}
\usepackage{amssymb}
\usepackage{mathrsfs}
\usepackage{amscd}
\usepackage{mathtools}
\usepackage{url}
\usepackage{pgfplots}
\pgfplotsset{compat=newest}
\usepackage{multirow}

\newcommand{\re}{\mathbb{R}}

\newcommand{\N}{\mathbb{N}}

\def\af{\alpha}

\def\rank{\mbox{rank}}

\newcommand{\st}{\mbox{s.t.}}

\newcommand{\reff}[1]{(\ref{#1})}

\newcommand{\mc}[1]{\mathcal{#1}}

\newcommand{\qmod}[1]{\mathit{QM}[#1]}

\newcommand{\pre}[1]{\mathit{Pre}[#1]}

\newcommand{\bdes}{\begin{description}}
\newcommand{\edes}{\end{description}}

\newcommand{\bal}{\begin{align}}
\newcommand{\eal}{\end{align}}

\newcommand{\bnum}{\begin{enumerate}}
\newcommand{\enum}{\end{enumerate}}

\newcommand{\bit}{\begin{itemize}}
\newcommand{\eit}{\end{itemize}}

\newcommand{\bea}{\begin{eqnarray}}
\newcommand{\eea}{\end{eqnarray}}
\newcommand{\be}{\begin{equation}}
\newcommand{\ee}{\end{equation}}

\newcommand{\baray}{\begin{array}}
\newcommand{\earay}{\end{array}}

\newcommand{\bsry}{\begin{subarray}}
\newcommand{\esry}{\end{subarray}}

\newcommand{\bca}{\begin{cases}}
\newcommand{\eca}{\end{cases}}

\newcommand{\bcen}{\begin{center}}
\newcommand{\ecen}{\end{center}}

\newcommand{\bbm}{\begin{bmatrix}}
\newcommand{\ebm}{\end{bmatrix}}

\newcommand{\bmx}{\begin{matrix}}
\newcommand{\emx}{\end{matrix}}

\newcommand{\bpm}{\begin{pmatrix}}
\newcommand{\epm}{\end{pmatrix}}

\newcommand{\btab}{\begin{tabular}}
\newcommand{\etab}{\end{tabular}}

\newtheorem{theorem}{Theorem}[section]

\newtheorem{prop}[theorem]{Proposition}

\newtheorem{corollary}[theorem]{Corollary}
\newtheorem{ass}[theorem]{Assumption}

\theoremstyle{definition}
\newtheorem{example}[theorem]{Example}

\newtheorem{algorithm}[theorem]{Algorithm}

\numberwithin{equation}{section}

\begin{document}

\title{Towards Global Solutions for Nonconvex Two-Stage Stochastic Programs: 
A Polynomial Lower Approximation Approach}

\author{Suhan Zhong \thanks{Department of Mathematics,
Texas A\&M University, College Station, TX 77843 (suzhong@tamu.edu).
} \and
Ying Cui \thanks{
Department of Industrial Engineering and Operations Research, 
University of California, Berkeley,  Berkeley, CA  94720 (yingcui@berkeley.edu). 
The author is partially supported by the National Science Foundation under grants CCF-2153352 and DMS-2309729, and the National Institutes of Health under grant R01CA287413-01.}
\and Jiawang Nie \thanks{Department of Mathematics,
University of California San Diego, La Jolla, CA 92093 (njw@math.ucsd.edu).
The author is partially supported by the National Science Foundation under grant  DMS-2110780.}}
\date{}

\maketitle

\begin{abstract}
This paper tackles the challenging problem of finding global optimal solutions for 
two-stage stochastic programs with continuous decision variables and nonconvex 
recourse functions. We introduce a two-phase approach. The first phase involves the
construction of a polynomial lower bound for the recourse function through a linear
optimization problem over a nonnegative polynomial cone. Given the complex structure 
of this cone, we employ semidefinite relaxations with quadratic modules to facilitate 
our computations. In the second phase, we solve a surrogate first-stage problem by
substituting the original recourse function with the polynomial lower approximation
obtained in the first phase. Our method is particularly advantageous for two reasons: 
 it not only generates global lower bounds for the nonconvex stochastic program, 
aiding in the certificate of global optimality for prospective solutions like 
stationary solutions computed from other methods, but it also
yields an explicit polynomial approximation for the recourse function through the solution of a linear conic optimization problem, 
where the number of variables is independent of the support of the underlying random vector.
Therefore,  our approach  is particularly suitable for the 
case where the random vector follows a continuous distribution or when dealing with 
a large number of scenarios. Numerical experiments are conducted to demonstrate the
effectiveness of our proposed approach. 
\end{abstract}
{\bf Keywords:} two-stage stochastic programs, 
polynomial optimization, nonconvex, global solutions

\noindent
{\bf MSC Classification:} 90C23 , 65K05 , 90C15

\section{Introduction}
\label{sec: introduction}

Two-stage stochastic programs (SPs) with recourse functions serve as a powerful framework 
for modeling  decision-making problems under uncertainty. In the first stage, 
``here-and-now'' decisions are made prior to the uncertainty being revealed. 
Following this, the second stage accommodates additional decisions, often contingent 
on the outcomes of the uncertainty, and are referred as ``recourse actions".
The goal of two-stage SPs is to determine  decisions that minimize the expected 
total cost. Mathematically, a two-stage SP with recourse functions is formulated as
\begin{equation}\label{eq:2stageSP}
\left\{\begin{array}{cl}
\min\limits_{x\in \re^{n_1}} & 
f(x): = f_1(x)+\mathbb{E}_{\mu}[f_2(x,\xi)]\\[0.1in]
\st & x\in X \, \coloneqq \, \{x\in\re^{n_1}:  g_{1,i}(x)\ge 0\,(i\in\mc{I}_1)\},
\end{array}\right.
\end{equation}
where $\xi\in \mathbb{R}^{n_0}$ is a random vector associated with the probability 
measure $\mu$ supported on 
\begin{equation}\label{eq:setS1} 
S\,\coloneqq \,\{\xi\in\re^{n_0}: g_{0,i}(\xi)\ge 0\, (i\in \mc{I}_0)\}, 
\end{equation}
and $f_2(x,\xi)$ is the so called {\it recourse function} 
given by:
\begin{equation}\label{eq:second_stage}
\,\left\{
\begin{array}{rl}
f_2(x,\xi) := \min\limits_{y\in \mathbb{R}^{n_2}} 
& F(x,y,\xi)\\[0.05in]
\st\, \,& y\in Y(x,\xi)\,\coloneqq\, \{ y\in\re^{n_2}: 
g_{2,i}(x,y,\xi)\ge 0\,(i\in\mc{I}_2) \}.
\end{array}
\right.
\end{equation}
Here \reff{eq:2stageSP}--\reff{eq:second_stage} satisfy the following assumption.
\begin{ass}
The index sets $\mc{I}_0,\mc{I}_1$ and  $\mc{I}_2$ are finite, potentially empty.
The functions $g_{0,i}:\,\re^{n_0}\to\re$ for each $i\in \mc{I}_0$;
$f_1,\, g_{1,i}: \,\re^{n_1}\to\re$ for $i\in \mc{I}_1$; and
$F,\, g_{2,i}: \,\re^{n_1}\times \re^{n_2}\times \re^{n_0}\to \re$ for $i\in\mc{I}_2$.
\end{ass}
As a versatile modeling paradigm, two-stage SPs 
have found applications across numerous domains, such as supply chain management
\cite{GLRS11,LAC12}, energy systems \cite{DJ22,MMP23}, and transportation planning
\cite{HouSun19,PaulZhang19}, among others. For a comprehensive understanding of this
subject matter, readers are referred to the monographs \cite{Birgebook,Shapiro21book} 
and references therein. 

\vskip 0.05in
When  $f$ is a convex function and $X$ is a convex set,
problem \reff{eq:2stageSP} is  convex. Numerical methods for solving convex two-stage 
SPs have been extensively studied. When $\xi$ follows a discrete distribution or is
approximated by sample averages, \reff{eq:2stageSP} simplifies to a convex deterministic
problem, enabling the application of the $L$-shaped method \cite{VanWets69,Wets84}, 
the (augmented) Lagrangian method \cite{ParpasRustem07}, and the progressive hedging
method \cite{Guignard03,Rockafellar91}. In instances where $\xi$ follows a continuous
distribution, one may either directly employ stochastic approximation or utilize sample
average approximation to recast it into a deterministic formulation, subsequently
applying the aforementioned methods. Under technical assumptions, the (sub)sequences
generated by these algorithms converge to globally optimal solutions to the convex SPs.

\vskip 0.05in
Many real-world applications feature two-stage SPs that are inherently nonconvex.
Examples include the two-stage stochastic interdiction problem \cite{cormican1998stochastic,hao2020piecewise} 
and the stochastic program with decision-dependent uncertainty \cite{GoelGrossmann04,Hellemo18,LiCui22,LiuCui20,GrigasQi21}. 
In fact, the recourse function in the form of \eqref{eq:second_stage} easily 
becomes nonconvex in the first-stage variable $x$, even in the simple situation 
where the second-stage problem is linearly parameterized by $x$:
\[
\,\left\{
\begin{array}{rl}
f_2(x,\xi) = \min\limits_{y\in \mathbb{R}^{n_2}} 
& [\,c(\xi) + C(\xi) x\,]^T y\\[0.05in]
\st\, \,& A(\xi) x + B(\xi) y \ge b(\xi).
\end{array}
\right.
\]
It is important to note that the nonconvexity in the above problem does not arise 
from the integrality of decision variables $y$, and thus techniques from 
mixed-integer programming are not applicable here. For such problems, the focus 
in the existing literature is primarily on the efficient computation of local solutions, such as stationary points \cite{Borges21, LiCui22,LiuCui20}. 
Generally, it is challenging to compute global optimal solutions of nonconvex 
two-stage SPs as well as to certify the quality of a given point in terms of its global optimality.  

\vskip 0.05in
The primary goal of the present paper is to design a relaxation approach that can
asymptotically solve problem \reff{eq:2stageSP} to global optimality, under the setting
that the recourse function $f_2$ is nonconvex in $x$. 
Throughout this paper, we consider the two-stage SP in the form of \eqref{eq:2stageSP} satisfying the following condition.
\begin{ass}
The functions $F(x,y,\xi)$, $\{g_{1,i}(x)\}$, $\{g_{0,i}(\xi)\}$ and $\{g_{2,i}(x,y,\xi)\}$ are all polynomials in terms of the arguments $(x,y, \xi)$.
\end{ass}

One major challenge in globally
solving \eqref{eq:2stageSP} stems from the typical lack of an explicit parametric
representation of the recourse function $f_2(x,\xi)$. To overcome this difficulty, 
we introduce a two-phase algorithm. In the first phase, we construct a parametric
function $p(x,\xi)$ that serves as a lower approximation of the recourse function 
$f_2(x,\xi)$ over $X\times S$, satisfying
\begin{equation}\label{eq:lower bound}
f_2(x,\xi) - p(x,\xi) \ge 0,\quad \forall \,x\in X,\,\xi\in S.
\end{equation}
In the second phase, we replace $f_2(x,\xi)$ in problem \reff{eq:2stageSP} with the
approximating function $p(x,\xi)$ and solve the corresponding surrogate problem to 
global optimality. Given that $p$ provides a lower approximation of $f_2$ on its domain,
the global optimal value computed from the surrogate problem must be a lower bound of 
the true optimal value of problem \reff{eq:2stageSP}. Consequently, this computed value
also provides an estimate of the distance from the objective value at a local 
solution/stationary point that is obtained by any other methods to the true global
optimal value. In addition, we design a hierarchical procedure to asymptotically 
diminish the gap between $f_2(x,\xi)$ and $p(x,\xi)$ (in the ${\cal L}^1$ space), 
thereby ensuring that the objective value obtained from the surrogate problem converges to the true global optimal value of \reff{eq:2stageSP}.

\vskip 0.05in
To achieve our goal of finding the global optimal solution of the nonconvex two-stage SP,
we leverage techniques from polynomial optimization. It is well known that under the
archimedean condition, a generic polynomial optimization problem can be 
solved to global optimality through a hierarchy of Moment-Sum-of-Squares (Moment-SOS)
relaxations \cite{Las01}; see, for example, the monographs
\cite{Las09book,Lau14,NieBook}. Specifically, let us denote
\begin{equation}\label{defn:F and K} 
\mc{F}\,\coloneqq\, \{(x,\xi)\in X\times S: Y(x,\xi) \neq \emptyset\} \quad \mbox{and}\quad
K \,\coloneqq\, \{(x,y,\xi): (x,\xi)\in \mc{F}, y\in Y(x,\xi)\}.
\end{equation}
Then for any $(x,\xi)\in \mc{F}$, the inequality \eqref{eq:lower bound} is equivalent to
\begin{equation}\label{eq:F-p>0:int}
F(x,y,\xi) - p(x,\xi)\ge 0, \quad \forall \, y\in Y(x,\xi).
\end{equation}
Assuming that the functions $F$ and $g_{2,i}$ for $i\in {\cal I}_2$ in \eqref{eq:second_stage} are polynomials over $(x,y,\xi)$, we construct a polynomial function $p(x,\xi)$ such that 
 $F(x,y,\xi)-p(x,\xi)$ is a nonnegative polynomial  over $K$.
Obviously there are infinitely many polynomials satisfying the above condition. In order to approximate the recourse function $f_2(x,\xi)$ as tight as possible, we seek the one that is closest to it from below under a prescribed metric. Specifically, 
letting $\mathscr{P}(K)$ be the set of polynomials in $(x,y,\xi)$ that are nonnegative on $K$ and
 $\nu$ be a probability measure supported on $\mc{F}$, we solve for a best polynomial lower approximating function via the following problem  
\begin{equation}\label{eq:maxim:int}
\left\{
\begin{array}{cl}
\max\limits_{p} & \int_{\mc{F}} p(x,\xi) {\tt d}\nu\\
\st & F(x,y,\xi) - p(x,\xi)\in\mathscr{P}(K).
\end{array}
\right.
\end{equation}
When the degree of the polynomial $p(x,\xi)$  is fixed, the above problem reduces to a linear conic optimization in the coefficients of $p$. A noteworthy benefit of  problem \eqref{eq:maxim:int} is that the size of the decision variables are determined merely by the dimensions of $(x,\xi)$ and the degree of the polynomial $p$, while remaining unaffected by the distribution of $\xi$ or the number of samples used to approximate $\xi$'s distribution. This becomes particularly advantageous when there is a large number of scenarios for $\xi$. Even more appealingly, if $\xi$ follows a continuous distribution, there is no necessity  to draw samples to approximate its distribution in order to compute $\mathbb{E}_\nu [\,p(x,\xi)\,]$; it can instead be computed  analytically through the moments of $\xi$.

\begin{center}
\begin{figure}[htb]
\begin{center}
\begin{minipage}{.35\textwidth}
\centering
\includegraphics[width=\textwidth]{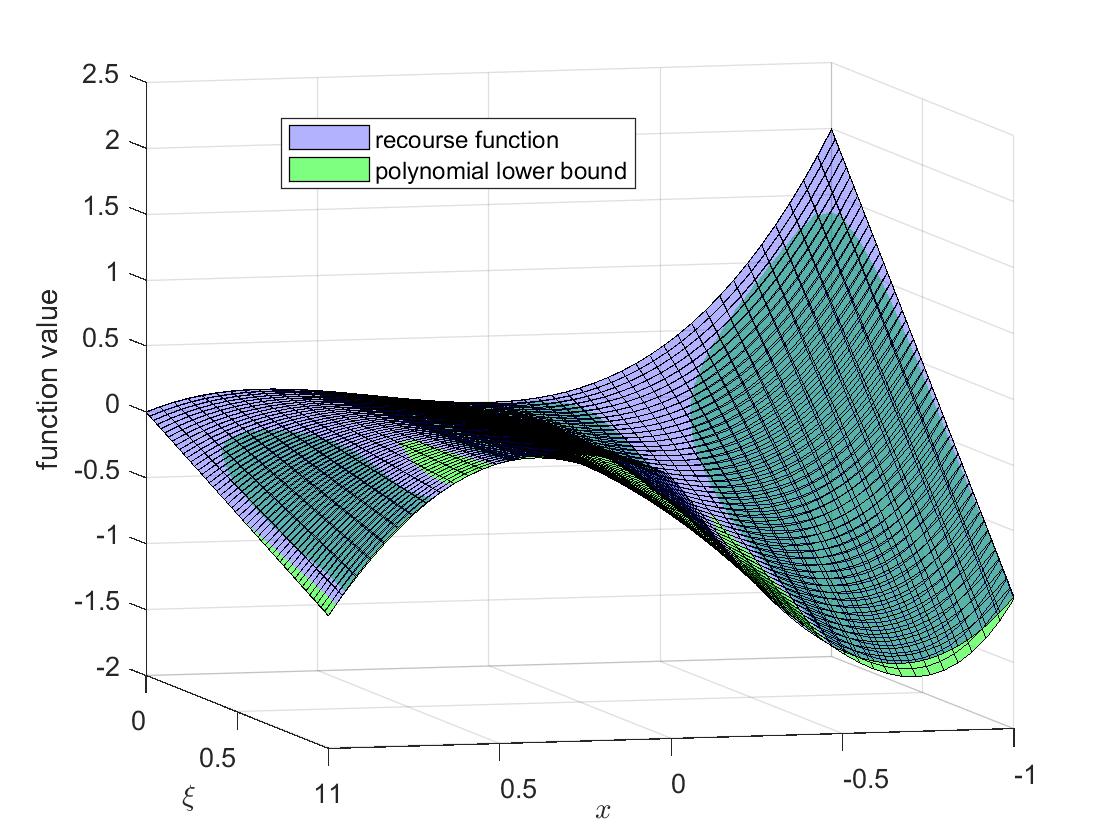}
\end{minipage}
\qquad
\begin{minipage}{.35\textwidth}
\centering
\includegraphics[width=\textwidth]{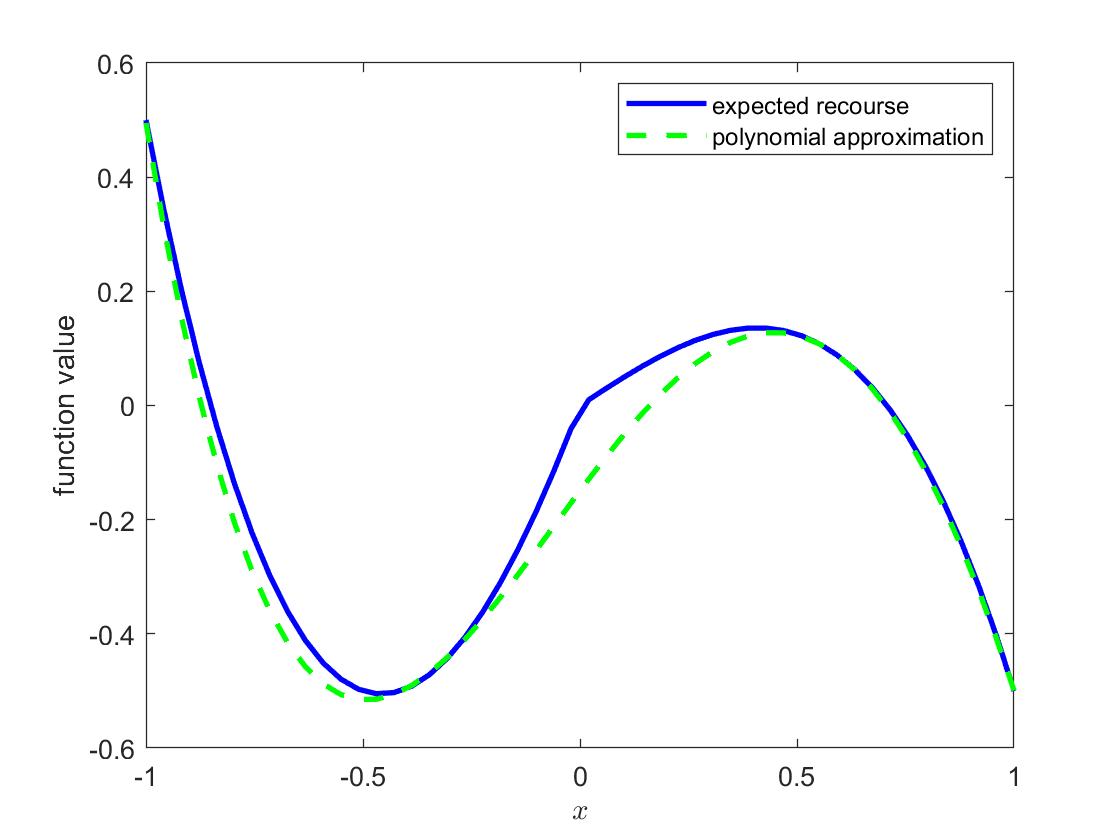}
\end{minipage}
\end{center}
\caption{An illustration of our approach. The left panel shows how a nonconvex recourse function $f(x,\xi)$ can be efficiently approximated from below by a polynomial $p(x,\xi)$. The right panel shows the expectation of the recourse and its approximation from the left (i.e., $\mathbb{E}_{\nu}[f(x,\xi)]$ vs $\mathbb{E}_{\nu}[p(x,\xi)]$) for a given measure $\nu$.}
\end{figure}
\end{center}

We outline the major advantages of our proposed approach below.
\begin{itemize}
\item[(a)]
Our method efficiently computes lower bounds for the global optimal value of problem  \reff{eq:2stageSP}, which can be particularly tight when the recourse function
is polynomial. These bounds can be used to certify the global optimality of prospective solutions like stationary solutions computed from other methods.
\item[(b)] 
The approach yields an explicit polynomial lower bound for the recourse function.
With certain assumptions of compactness and continuity, these polynomials can achieve an arbitrary level of accuracy in the $\mc{L}^1$ space relative to a given probability measure.
\item[(c)] The number of variables in problem \reff{eq:maxim:int} is independent 
of the distribution of $\xi$. 
Therefore, our approach is especially beneficial in instances where $\xi$ follows a continuous distribution or is approximated by a large 
number of scenarios.
\end{itemize}
The rest of this paper is organized as follows. Some notation and basic knowledge on 
polynomial optimization is introduced right after.
In Section~\ref{sec:plb}, we discuss the construction of polynomial lower approximation of the recourse function via linear conic 
optimization. 
Utilizing the derived polynomial lower approximating functions, we develop algorithms to approximately solve nonconvex two-stage SPs 
in Section~\ref{sc:paa},  and study their convergent properties. 
In Section~\ref{sc:lco}, 
the Moment-SOS relaxation methods are introduced to solve the subproblems arising from the 
algorithms in the previous section. Some numerical results are given 
in Section~\ref{sc:numexp}. The paper ends with a concluding section.

\section*{Notation and Preliminaries}
The symbol $\mathbb{R}$ denotes the set of real numbers and $\N$ denotes the set of nonnegative 
integers. The notation $\re^n$ (resp., $\N^n$) stands for the set of $n$-dimensional vectors with entries in 
$\re$ (resp., $\N$). For $t\in \mathbb{R}$, $\lceil t\rceil$ denotes the smallest integer that is not smaller 
than $t$. For an integer $k>0$, denote $[k] \coloneqq \{1,\cdots,k\}$. For a vector $v\in\re^n$, 
we use $\|v\|$ to denote its Euclidean norm. The superscript $^T$ denotes the transpose 
of a matrix or vector. Let $\Omega_1$ and $\Omega_2$ be two sets. Their Cartesian product is denoted as
$\Omega_1\times \Omega_2 \,  \coloneqq  \, \{(v_1, v_2): \, v_1\in \Omega_1, \, v_2\in \Omega_2\}$.
Let $\nu$ be a probability measure supported on $\Omega_1$ and $\mc{L}^1(\nu)$ denote 
the set of functions $f:\Omega_1\to \re$ such that $\int_{\Omega_1}|f|{\tt d}\nu<\infty$.
A matrix $A\in\re^{n\times n}$ is said to be positive semidefinite, denoted as $A\succeq 0$,  
if $v^TAv\ge 0$ for all $v\in\re^n$. If $v^TAv>0$ for every nonzero vector $v\in\re^n$, 
then $A$ is positive definite, written as $A\succ 0$.
Let $w = (w_1,\ldots, w_\ell)$ be a vector of variables. We use $\re[w]$ to denote the ring of real 
polynomials in $w$. Then $\re[w]_d\subseteq\re[w]$ is the set of real polynomials with degrees no 
more than $d$. For a polynomial $f(x,y,\xi)$, its total degree is denoted by $\deg(f)$. 
We use $\deg_x(f)$ (resp., $\deg_y(f)$, $\deg_{\xi}(f)$) to denote its partial degree in $x$ (resp., $y$, $\xi$).
For a tuple of polynomials $h = (h_1,\ldots, h_m)$, the notation $\deg(h)$ represents the highest degree 
among all $h_i$'s. For a monomial power $\af  \coloneqq  (\af_1, \ldots, \af_\ell)\in \N^\ell$, denote
\[ w^\af \, \coloneqq  \, w_1^{\af_1} \cdots  w_\ell^{\af_\ell} , 
\quad\mbox{with}\quad |\af| \, \coloneqq \, \af_1 + \cdots + \af_\ell. \]
For a degree $d$, denote the set of monomial powers in $w$ as 
$\N_d^\ell\,\coloneqq\, \{\alpha\in\N^\ell: |\alpha|\le d\}$. The notation
\[ [w]_d \, \coloneqq  \, \bbm 1 & w_1 & \cdots & w_\ell & (w_1)^2 & w_1w_2
   & \cdots & (w_\ell)^d \ebm^T, \]
denotes the monomial vector with the highest degree $d$ and ordered alphabetically.

A polynomial $p\in\mathbb{R}[w]$ is said to be a sum-of-squares (SOS) if it can be expressed as 
$p = p_1^2 + \cdots + p_t^2$ for some  $p_1,\ldots, p_t\in\re[w]$. The set of all SOS polynomials 
in $w$ is denoted by $\Sigma[w]$. Its $d$th degree truncation is denoted by 
$\Sigma[w]_d\coloneqq \Sigma[w]\cap \re[w]_d$.
Let $h = (h_1,\ldots, h_m)$ be a tuple of polynomials and define $\Omega = \{w\in\re^\ell: h(w) \ge 0\}$.
We denote the nonnegative polynomial cone over $\Omega$ as
\[ \mathscr{P}(\Omega) \,\coloneqq\, \{p\in\re[w]: p(w)\ge 0,\forall w\in \Omega\}. \]
For every degree $d$, $\mathscr{P}_d(\Omega) \,\coloneqq\, \mathscr{P}(\Omega)\cap \re[w]_d$.
The preordering of $h$ is given as
\begin{equation}\label{eq:pre[h]}
\pre{h} \,\coloneqq \, \sum\limits_{J\subseteq [m]} \Big(\prod\limits_{i\in J}h_i\Big)\cdot \Sigma[w].
\end{equation}
Clearly, $\pre{h}\subseteq \mathscr{P}(\Omega)$. Interestingly, when $\Omega$ is compact, every
polynomial that is positive on $\Omega$ belongs to $\pre{h}$. This conclusion is referenced as 
{\it Schmudgen's Positivstellensatz} \cite{PutinarPositive}.
The quadratic module of $h$ is a subset of $\pre{h}$, which is defined as
\[ \qmod{h} \,  \coloneqq  \,  \Sigma[w] + h_1\cdot\Sigma[w] + \cdots + h_m \cdot\Sigma[w] . \]
Its $k$th order truncation is given as 
\begin{equation}\label{eq:qmod}
\qmod{h}_{2k} \, \coloneqq  \, \Sigma[w]_{2k} + 
h_1\cdot\Sigma[w]_{2k-\deg(h_1)}+ \cdots+ h_{m} \cdot\Sigma[w]_{2k-\deg(h_m)}.
\end{equation}
When $\Omega$ is compact, $\qmod{h}$ and each $\qmod{h}_{2k}$ are closed convex cones. 
For every $k$ such that $2k\ge \deg(h)$, the nested containment 
relation holds that
\[
\qmod{h}_{2k} \,\subseteq\, \qmod{h}_{2k+2} \,\subseteq\, 
\cdots\, \subseteq\, \qmod{h} \,\subseteq\, \pre{h}\,\subseteq\,\mathscr{P}(\Omega).
\]
In particular,  $\qmod{h}$ is said to be {\it archimedean} if there exists $q\in\qmod{h}$ such that 
$q(w)\ge 0$ determines a compact set. Suppose $\qmod{h}$ is archimedean. Every polynomial that is 
positive on $\Omega$ must be contained in $\qmod{h}$. This conclusion is called 
{\it Putinar's Postivstellensatz} \cite{PutinarPositive}. It is clear that $\Omega$ is compact when 
$\qmod{h}$ is archimedean. Conversely, if $\Omega$ is compact, $\qmod{h}$ may not be archimedean.
In this case, we can always find a sufficiently large $R>0$ such that $\Omega$ is contained in 
$\{w: R-\|w\|^2\ge 0\}$ and that $\qmod{\tilde{h}}$ is archimedean for $\tilde{h} = (h, R-\|w\|^2)$.

For an integer $k\ge 0$, a real vector $z = (z_{\alpha})_{\alpha\in\mathbb{N}_{2k}^{\ell}}$ is said to be a
{\it truncated multi-sequence} (tms) of $x$ with degree $2k$. For a polynomial 
$p(x) = \sum_{ \af \in \N^{\ell}_{2k} } p_\af x^\af$, denote the bilinear operation in $p$ and $z$ as:
\be \label{eq:<p,z>}
\langle p, z \rangle \,  \coloneqq  \,
{\sum}_{ \af \in \N^\ell_{2k} } p_\af z_\af.
\ee
For a polynomial $q\in\re[x]_{2t}$ with $t\le k$, the $k$th order {\it localizing matrix} of $q$ and $z$
is the symmetric matrix $L_{q}^{(k)}[z]$ that satisfies
\be \label{eq:Lq[z]}
\langle qa^2, z \rangle  \, =  \,
vec(a)^T \big( L_{q}^{(k)}[z]  \big) vec(a)
\ee
for each polynomial $a(x) = vec(a)^T[x]_s$ with $s\le k - t$. When $q=1$ being the constant one polynomial, 
$L_q^{(k)}[z]$ becomes the $k$th order \textit{moment matrix} $M_t[z]  \coloneqq L_{1}^{(k)}[z]$.
Quadratic modules and their dual cones play a critical role in polynomial optimization.
Recently, polynomial optimization has been actively studied in \cite{HNpo23,HuangNie23,MaiLas23,QuTang22}.
We refer to monographs \cite{Las09book,Lau14,NieBook} for comprehensive results in polynomial optimization.

\section{Lower Approximations of Recourse Functions via Polynomials}
\label{sec:plb}

This section is devoted to the phase one of our approach on the construction of a polynomial lower approximation of the 
(nonconvex) recourse function $f_2(x,\xi)$ over ${\cal F}$, under the assumption that the functions $F(x,y,\xi)$, $\{g_{2,i}(x,y,\xi)\}$ and $\{g_{0,i}(\xi)\}$ in problems \eqref{eq:2stageSP} and \eqref{eq:setS1} are polynomials.

\subsection{Linear conic optimization}
\label{ssc:lco}

In this subsection, we discuss how to solve problem \eqref{eq:maxim:int}.
This is a linear conic optimization problem whose decision variable is the coefficient vector of $p(x,\xi)$.
We start with a toy example.
\begin{example}\label{eq:lco}
Let $x,y,\xi\in\re$ and 
\[
F(x,y,\xi) = (x+y-\xi)^2,\quad X = S =\re,\quad  Y(x,\xi)=Y=\re. 
\]
Obviously $\mc{F}=\re^2$ and  $K=\re^3$. We take 
$\nu$ as the standard normal distribution on $\mathbb{R}^2$ and $p(x,\xi)$ as 
a quadratic polynomial in the form of
\[
p(x,\xi) \,=\,  p_{00} + p_{10}x + p_{01}\xi + p_{20}x^2 + p_{11}x\xi + p_{02}\xi^2.
\]
Since $ \int_{\mc{F}} x{\tt d}\nu = \int_{\mc{F}} \xi{\tt d}\nu = \int_{\mc{F}} x\xi{\tt d}\nu = 0$
and $\int_{\mc{F}}1{\tt d} \nu = \int_{\mc{F}} x^2{\tt d}\nu = \int_{\mc{F}} \xi^2{\tt d}\nu = 1$, we have
\[ \int_{\mc{F}} p(x,\xi){\tt d}\nu \,=\, p_{00} + p_{20}+p_{02}. \]
 In addition, since
$\mathscr{P}_2(\re^3) = \Sigma[x,y,\xi]_2$, we have that
\[
F(x,y,\xi) - p(x,\xi) \,=\, \bbm 1\\ x\\ y\\ \xi\ebm^T
\left[\begin{array}{rrrr}
-p_{00} & -0.5 p_{10} & 0 & -0.5p_{01}\\ 
-0.5p_{10}  & 1-p_{20} &  1 & -1-0.5p_{11}\\
0 & 1 & 1 & -1\\ 
-0.5p_{01} & -1-0.5p_{11} & -1 & 1-p_{02}
\end{array}\right]
\bbm 1\\ x\\y \\ \xi\ebm
\]
is nonnegative on $\re^3$ if and only if the above coefficient matrix is positive semidefinite. 
This is satisfied when all coefficients of $p(x,\xi)$ are zeros, i.e.,  $p=0$ is the identically zero polynomial.
\end{example}

In general, even if all the functions $F(x,y,\xi)$ and $g_{2,i}(x,y,\xi)$ for $i\in {\cal I}_2$  in 
\eqref{eq:second_stage} are polynomials, the value function $f_2(x,\xi)$ may not be continuous, 
as can be seen from the following example ($x,y,\xi$ are all univariate):
\[ 
\left(\begin{array}{ll}
f_2(x,\xi) = & \displaystyle \min_{y} \; y \\
& \mbox{s.t.}\quad xy=0,\;  -1\leq y\leq \xi^2.
\end{array}\right)   = \left\{\begin{array}{ll} -1 &\mbox{if $x = 0$,}\\[0.02in]
0 &\mbox{if $x\neq 0$.}
\end{array}\right. 
\] 
Additional assumptions are needed to make the function $f_2$ continuous, such as the 
 {\it restricted inf-compactness} 
 condition \cite[Definition 3.13]{GuoLinYeZhang} together with some constraint qualifications.
We refer to \cite[\S6.5.1]{Clarke} and \cite{GauvinDubeau,Rochafellar82} for more details on these results.
When $\mc{F}$ is compact and the value function $f_2(x,\xi)$ is continuous,
the objective in \reff{eq:maxim:int} is bounded from above and its optimal value equals the integral of 
$f_2(x,\xi)$ with respect to $\nu$. We formally state the results below. 
\begin{theorem}
\label{thm:gam_equal}
Assume $\mc{F}$ is compact and $f_2(x,\xi)$ is continuous on $\mc{F}$. For a given probability 
measure $\nu$ supported on $\mc{F}$, the objective in \reff{eq:maxim:int} is bounded from above 
on its feasible region and the optimal value equals $\int_{\mc{F}} f_2(x,\xi){\tt d}\nu$.
\end{theorem}
\begin{proof}
Under the given assumptions, the integral $\int_{\mc{F}} f_2(x,\xi){\tt d}\nu$ is finite and it is an upper 
bound for the optimal value of \reff{eq:maxim:int}. Let $\varepsilon>0$ be an arbitrarily small scalar.
By Weierstrass approximation theorem \cite[Theorem~7.26]{RudinMathAnalysis}, there is a polynomial 
$q_{\varepsilon}\in\re[x,\xi]$ such that 
\[
|f_2(x,\xi) - q_{\varepsilon}(x,\xi)| \,\le\, \varepsilon,
\quad \forall (x,\xi)\in \mc{F}.
\]
Let $\tilde{q}_{\varepsilon}(x,\xi)\coloneqq q_{\varepsilon}(x,\xi)-\varepsilon$. It is feasible for 
\reff{eq:maxim:int} because for every triple $(x,y,\xi)\in K$, we have
\[
F(x,y,\xi) - \tilde{q}_{\varepsilon}(x,\xi) \,\ge\, (f_2(x,\xi) - q(x,\xi)) +\varepsilon \,\ge\, 0.
\]
In addition, for the given probability measure $\nu$, since $\int_{\mc{F}} 1{\tt d}\nu = 1$,
it holds that
\[
\int_{\mc{F}} | f_2(x,\xi) - \tilde{q}_{\varepsilon}(x,\xi) | {\tt d}{\nu}
\,\le\, \max\limits_{(x,\xi)\in\mc{F}}\, 
|f_2(x,\xi) - q_{\varepsilon}(x,\xi)| + \varepsilon\,\le\, 2\varepsilon.
\] 
Since $\varepsilon$ can be arbitrarily small, there exists a sequence of optimizing polynomials 
converging to $f_2(x,\xi)$ in $\mc{L}^1(\nu)$.  Hence, their integrals converge to 
$\int_{\mc{F}}f_2(x,\xi){\tt d}\nu$.
\end{proof}

When the recourse function $f_2(x,\xi)$ is itself a polynomial, problem \reff{eq:maxim:int} has a 
global optimal solution $f_2(x,\xi)$. If, however,  the function $f_2(x,\xi)$ is not a polynomial, 
we can construct a sequence of approximating polynomial functions $\{p^{(k)}\}_{k = 1}^{\infty}$, each serving as a  
 lower bound for $f_2(x,\xi)$ over ${\cal F}$. Furthermore, the integral
$\int_{\mc{F}} p^{(k)}(x,\xi)\, {\tt d}\nu$ converges to the optimal value of 
\reff{eq:maxim:int} as $k\to \infty$. In Section~\ref{sc:lco}, we will discuss how 
 to compute such a convergent polynomial sequence numerically.

Suppose $\{p^{(k)}(x,\xi)\}_{k=1}^{\infty}$ is an optimizing sequence of \reff{eq:maxim:int}, 
i.e., each of them is feasible to \reff{eq:maxim:int} and  
$\displaystyle\lim_{k\to \infty}\int_{\mc{F}}p^{(k)}(x,\xi){\tt d}\nu = \int_{\mc{F}}f_2(x,\xi){\tt d}\nu$.
Then the term $\mathbb{E}_{\mu}[f_2(x,\xi)]$, which is the expectation of the recourse function
 in the first stage problem \eqref{eq:2stageSP},   should be well approximated by 
$\mathbb{E}_{\mu}[p^{(k)}(x,\xi)]$  when $k$ is sufficiently large.
The accuracy of the estimation depends on the selection of the probability measure $\nu$. For instance, if $\nu$ 
is the uniform distribution over  $\mc{F}$, then \reff{eq:maxim:int} finds
a lower approximating function that uniformly approximates $f_2(x,\xi)$ across $\mc{F}$. If we define
$\nu\coloneqq \delta_{\hat{x}}\times \mu$,  where $\delta_{\hat{x}}$ is a Dirac measure centered at $\hat{x}\in X$,   and $S$ denotes the projection of 
$\mc{F}$ onto the $\xi$-plane, then the objective of \reff{eq:maxim:int} reduces to 
\[
\int_{\mc{F}} p(x,\xi){\tt d}(\delta_{\hat{x}}\times \mu)
\,=\, \int_{S} p(\hat{x},\xi){\tt d}\mu\,=\, 
\mathbb{E}_{\mu}[p(\hat{x}, \xi)].
\]
Solving \reff{eq:maxim:int} gives an accurate evaluation of $\mathbb{E}_{\mu}[p(x,\xi)]$ 
at the point $x=\hat{x}$. In practice, we can strategically modify the measure $\nu$ 
to enhance the approximation of the original function in specific areas. Further discussions of this approach are given in the next section.

The requirement for $f_2(x,\xi)$ being continuous over $\mc{F}$ can be relaxed to being integrable with respect to the Lebesgue-Stieltjes measure $\nu$. This relaxed condition allows for the inclusion of functions that may possess discontinuities yet remain integrable. 
The formal statement and proof of this relaxation are given in the following corollary.

\begin{corollary}\label{cor:contin_conv}
(a) If $f_2(x,\xi)$ is a polynomial, then it must be a global  optimal solution of \reff{eq:maxim:int}.

\noindent
(b) Suppose $\mc{F}$ is compact and $\nu$ is a Lebesgue-Stieltjes probability measure supported 
on $\mc{F}$. If $f_2(x,\xi)\in \mc{L}^1(\nu)$, then the problem \reff{eq:maxim:int} is bounded from above,  and its 
optimal value is equal to $\int_{\mc{F}}f_2(x,\xi){\tt d}\nu$.
\end{corollary}
\begin{proof}
Part (a) is obvious. For part (b), when $\mc{F}$ is compact and $\nu$ is a Lebesgue-Stieltjes measure, the set of continuous 
functions is dense in $\mc{L}^1(\nu)$. Therefore,  the result can be proved via similar arguments in the proof of 
Theorem~\ref{thm:gam_equal}.
\end{proof}
We would like to highlight that using  SOS techniques to lower approximate  nonsmooth functions has been  extensively studied in the existing literature for various applications. In particular, when $\nu$ is the Lebesgue measure, 
the asymptotical convergence of the polynomial lower approximating  functions 
 towards different target functions are well studied under proper compact and semicontinuity assumptions. 
For example, 
the readers can find from \cite[Theorem~1]{HenLas12} for the approximation of eigenvalue functions in robust control, 
\cite[Theorem~3.2]{HessHenLas16} for the spectral abscissa, and 
\cite[Theorem~1]{HenPau17} for the value function in the optimal control.

\subsection{A special case: $\xi$ has a finite support}
\label{ssc:sc}

When the random vector $\xi$ has a finite support, say $S = \{\xi^{(1)},\ldots, \xi^{(r)}\}$, we may approximate the recourse function $f_2(x,\xi)$  at each $\xi^{(i)}$ individually by a polynomial merely in terms of $x$  to enhance the quality of the overall approximations. 
Specifically, assume
\begin{equation}\label{eq:disS}
\mu \,=\, \lambda_1 \delta_{\xi^{(1)}} + \lambda_2 \delta_{\xi^{(2)}} +\cdots + 
\lambda_r\delta_{\xi^{(r)}},
\end{equation}
where each $\lambda_i>0$ and that $\lambda_1+\lambda_2+\cdots +\lambda_r = 1$.
In this setting,  the expected recourse can be expressed as 
\[
\mathbb{E}_{\mu}[f_2(x,\xi)] \,=\, \lambda_1 f_2(x,\xi^{(1)}) + \lambda_2f_2(x,\xi^{(2)})+
\cdots + \lambda_r f_2(x,\xi^{(r)}).
\]
In the above, every $f_2(x,\xi^{(i)})$ is a function only dependent on $x$. Note
$f_2(x,\xi^{(i)}) \le F(x,y,\xi^{(i)})$ for every $y\in Y(x,\xi^{(i)})$. Since $F(x,y,\xi^{(i)})$ is 
a polynomial, when $X$ is compact, the function $f_2(x,\xi^{(i)})$ is bounded from above over the set 
\begin{equation}\label{eq:F_i}
\mc{F}_i \coloneqq \{x\in X: Y(x,\xi^{(i)})\not=\emptyset\}.
\end{equation}
The feasible region in \reff{defn:F and K} becomes $\mc{F} = \bigcup_{i = 1}^r\mc{F}_i\times \{\xi^{(i)}\}$. If for every $i\in[r]$, we can find a polynomial $p_i\in\re[x]$ such that 
\begin{equation}\label{eq:f2-pi>0inFi}
f_2(x,\xi^{(i)}) - p_i(x)\,\ge\, 0,\quad \forall x\in \mc{F}_i, 
\end{equation}
then  a lower approximating function for the expected recourse can be constructed as 
\begin{equation}\label{eq:lbf:dis}
p(x)\,\coloneqq\, \lambda_1 p_1(x)+ \lambda_2p_2(x)+\cdots+\lambda_r p_r(x).
\end{equation}
Consequently,  
$ \mathbb{E}_{\mu}[f_2(x,\xi)]-p(x)\,\ge\, 0$ for any $x\in X$.
Such polynomials $p_i(x)$ can be solved via linear conic optimization problems
similarly as in the previous subsection. Let $\nu_i$ be a probability measure supported on $\mc{F}_i$
and denote the feasible region
\begin{equation}\label{eq:Ki}
K_i \,\coloneqq\, \big\{(x,y): x\in\mc{F}_i,\, y\in Y(x,\xi^{(i)})\big\}.
\end{equation}
Consider the optimization problem
\begin{equation}\label{eq:maxim:int:dd}
\left\{\begin{array}{cl}
\max\limits_{p_i\in\re[x]} & \int_{\mc{F}_i} p_i(x) {\tt d}\nu_i\\
\st & F(x, y, \xi^{(i)}) - p_i(x) \in \mathscr{P}(K_i)^{x,y},
\end{array}\right.
\end{equation}
where $\mathscr{P}(K_i)^{x,y}\coloneqq\{q\in\re[x,y]: q(x,y)\ge 0,\forall (x,y)\in K_i\}$ is the 
nonnegative polynomial cone. To emphasize $\mathscr{P}(K_i)^{x,y}\subseteq \re[x,y]$, 
we add the superscript $^{x,y}$ to distinguish it from $\mathscr{P}(K)\subseteq \re[x,y,\xi]$.
Clearly, every feasible polynomial of \reff{eq:maxim:int:dd} satisfies \reff{eq:f2-pi>0inFi}.
Problem \reff{eq:maxim:int:dd} aims to find the best polynomial lower approximating function 
of $f_2(x,\xi^{(i)})$ such that
\[
\int_{\mc{F}_i} |f_2(x,\xi^{(i)}) - p_i(x)|{\tt d}\nu_i\,=\, 
\int_{\mc{F}_i} f_2(x,\xi^{(i)}){\tt d}\nu_i - \int_{\mc{F}_i} p_i(x){\tt d}\nu_i
\]
is minimized. Compared to problem \reff{eq:maxim:int}, problem \reff{eq:maxim:int:dd} has a smaller number 
of variables, which is expected to be easier to solve in practice. It has computational advantages when the cardinality of
the support set $S$ is small but the dimension for the random vector $\xi$ is large. Indeed, 
to solve for a polynomial lower bound function of degree $d$, the number of variables in 
\reff{eq:maxim:int} is $\binom{n_0+n_1+d}{d}$ and the number of variables in \reff{eq:maxim:int:dd} 
is $\binom{n_1+d}{d}$. In applications, the finite support $S$ is usually not given directly 
but is approximated by a large number of samples. In this case, we can apply 
the method proposed in \cite{nie21lossfunc} to find a finite set $\tilde{S}$ that is close to $S$. 
A group of lower approximating functions $\tilde{p}_i(x)$ can be similarly computed by solving 
\reff{eq:maxim:int:dd} with respect to each scenario in $\tilde{S}$. When $\tilde{S}$ is sufficiently 
close to $S$, such $\tilde{p}_i(x)$ can also be used to form a good approximation of the recourse 
function.

Under some compact and continuity assumptions, we can obtain similar 
results to Theorem~\ref{thm:gam_equal}.
\begin{theorem}
\label{thm:gam_equal_2}
Assume $\mc{F}_i$ is compact and $f_2(x,\xi^{(i)})$ is continuous on $\mc{F}_i$.
For a given probability measure $\nu_i$ supported on $\mc{F}_i$, problem \reff{eq:maxim:int:dd} is 
bounded from above and its optimal value is $\int_{\mc{F}_i} f_2(x,\xi^{(i)}){\tt d}\nu_i$.
\end{theorem}
\begin{proof}
Under given assumptions, the integral $\int_{\mc{F}_i} f_2(x,\xi){\tt d} \nu_i$ is finite and we have
$\int_{\mc{F}_i} f_2(x,\xi){\tt d} \nu_i\ge \int_{\mc{F}_i} p(x){\tt d}{\nu_i}$ for every 
feasible polynomial of \reff{eq:maxim:int:dd}. By Weierstrass approximation theorem 
\cite[Theorem~7.26]{RudinMathAnalysis}, for every $\varepsilon>0$, there exists a real 
polynomial $q_{\varepsilon}(x)$ such that 
\[
|f_2(x,\xi^{(i)}) - q_{\varepsilon}(x)|\,\le\, \varepsilon,\quad \forall x\in\mc{F}_i.
\]
Let $\tilde{q}_{\varepsilon}(x)\coloneqq q_{\varepsilon}(x)-\varepsilon$. It is feasible for
\reff{eq:maxim:int:dd} and satisfies
\[
\int_{\mc{F}_i} |f_2(x,\xi^{(i)}) - \tilde{q}_{\varepsilon}(x)|{\tt d}\nu_i
\,\le\, \max\limits_{x\in\mc{F}_i}\, |f_2(x, \xi^{(i)}) - q_{\varepsilon}(x)|+\varepsilon
\,\le\, 2\varepsilon.
\] 
Since $\varepsilon$ can be arbitrarily small, there exists a sequence of feasible optimizing
polynomials that converges to $f_2(x,\xi^{(i)})$ in $\mc{L}^1(\nu_i)$, with their integrals converging
to $\int_{\mc{F}_i} f_2(x,\xi^{(i)}){\tt d}\nu_i$.
\end{proof}
As in Corollary~\ref{cor:contin_conv}, the continuous assumption of $f_2(x,\xi^{(i)})$ can be 
relaxed when $\nu_i$ is a Lebesgue-Stieltjes measure.
\begin{corollary}
(a) If $f_2(x,\xi^{(i)})$ is a polynomial, then it must be an optimizer of \reff{eq:maxim:int:dd}.

\noindent
(b) Suppose $\mc{F}_i$ is compact and $\nu_i$ is a Lebesgue-Stieltjes measure supported on
$\mc{F}_i$. If $f_2(x,\xi)\in \mc{L}^1(\nu_i)$, then problem \reff{eq:maxim:int:dd} is bounded from above and
its optimal value is $\int_{\mc{F}_i}f_2(x,\xi^{(i)})$. 
\end{corollary}

\subsection{Conditions on tight lower bounds}
\label{ssc:tightbd}
A polynomial lower approximating function $p(x,\xi)$ is said to be a {\it tight} approximation of 
$f_2(x,\xi)$ on ${\cal F}$ with respect to the metric $\nu$ if 
$\int_{\mc{F}}|f_2(x,\xi)-p(x,\xi)|{\tt d\nu} = 0$. This particularly happens when $f_2(x,\xi)$ is itself  a polynomial.
It is thus an interesting question to understand the  conditions under which the recourse function is a polynomial.
For the two-stage SP \reff{eq:2stageSP}, denote the tuple of constraining polynomials as
\[
\tilde{g}(x,y,\xi)\,\coloneqq\, \big ((g_{0,i}(\xi))_{i\in\mc{I}_0},\quad  (g_{1,i}(x))_{i\in\mc{I}_1},
\quad (g_{2,i}(x,y,\xi))_{i\in\mc{I}_2}\big).
\]
It is clear that $K = \{(x,y,\xi): \tilde{g}(x,y,\xi)\ge 0\}$.
For convenience, we assume $[m]\coloneqq \mc{I}_0\cup\mc{I}_2\cup \mc{I}_2$ and use 
$\tilde{g}_i$ to denote the $i$th component of $\tilde{g}$. Then the preordering of $\tilde{g}$
can be written as
\[
\pre{\tilde{g}}\,\coloneqq\, \sum\limits_{J\subseteq [m]}
\Big(\prod\limits_{i\in J}\tilde{g}_i(x,y,\xi) \Big)\cdot \Sigma[x,y,\xi].
\]
Clearly, every polynomial in $\pre{\tilde{g}}$ is nonnegative on $K$.

First, we consider the relatively easy  case where \reff{eq:second_stage} is
an unconstrained optimization problem, i.e., $\mc{I}_2=\emptyset$ and  
 $F(x,y,\xi)$ is a quadratic function in $y$.
\begin{example}
Given $(x,\xi)$, suppose the second-stage problem takes the form of 
\[
f_2(x,\xi) = \left[\,\min\limits_{y\in\re^{n_2}}\quad  F(x,y,\xi) \,=\, \frac{1}{2}y^TAy +b(x,\xi)^Ty\,\right],
\]
where $A$ is a symmetric positive definite matrix. Since the objective function is strongly convex in $y$, 
we can solve for its unique optimizer $y^* = -A^{-1}b(x,\xi)$ from the 
first-order optimality condition $\nabla_y F(x,y^*,\xi) = Ay^* + b(x,\xi) = 0$.
This leads to the polynomial recourse function 
\[
f_2(x,\xi) = -\frac{1}{2}\,b(x,\xi)^TA^{-1}b(x,\xi).
\]
One can easily verify that  $F-f_2$ is a SOS polynomial, i.e.,
\[
F(x,y,\xi) - f_2(x,\xi) \,=\,  
\frac{1}{2}\, \big(y-A^{-1}b(x,\xi)\big)^TA \big(y-A^{-1}b(x,\xi)\big).
\]
In particular, for given $(x,\xi)$, the SOS polynomial on the right hand side can always achieve its 
global minimum at some $y\in\re^{n_2}$.
\end{example}
The above example motivates us to derive sufficient conditions of polynomial recourse functions
with SOS polynomial cones and preorderings, as stated in the following theorem.
\begin{theorem}\label{thm:poly_recourse}
Suppose that there exists a polynomial $q\in\pre{\tilde{g}}$ such that $F - q\in\re[x,\xi]$ and the set
\begin{equation}\label{eq:Vq}
\mc{V}_q(x,\xi)\,\coloneqq\, \{y\in\re^{n_2}: q(x,y,\xi) = 0\}
\end{equation}
is nonempty for every $(x,\xi)\in \mc{F}$. Then the recourse function of 
\reff{eq:2stageSP} satisfies $f_2(x,\xi)= F(x,y,\xi)-q(x,y,\xi)$ for any $(x,\xi)\in \mc{F}$.
\end{theorem}
\begin{proof}
Let $p\coloneqq F - q \in\re[x,y]$.
For given $(\hat{x},\hat{\xi})\in\mc{F}$, we have
\[ f_2(\hat{x},\hat{\xi}) - p(\hat{x},\hat{\xi}) 
\,=\, \min\limits_{y\in Y(\hat{x},\hat{\xi})} F(\hat{x},y,\hat{\xi}) - p(\hat{x},\hat{\xi}) 
\,=\, \min\limits_{y\in Y(\hat{x},\hat{\xi})} q(\hat{x},y,\hat{\xi}). \]
Notice that $K$ is a lifted set of $\mc{F}$ and $Y(x,\xi)$.
Since $K$ is determined by $\tilde{g}\ge 0$ and $q\in\pre{\tilde{g}}$, it holds that
\[
\min\limits_{y\in Y(\hat{x},\hat{\xi})} q(\hat{x},y,\hat{\xi}) 
\,\ge\, \min\limits_{(x,y,\xi)\in K} q(x,y,\xi)\,\ge \,0.
\]
In fact, $q(\hat{x}, y,\hat{\xi})=0$ can always be achieved since $\mc{V}_q(x,\xi)$ 
is nonempty for every $(x,\xi)\in\mc{F}$. The above arguments work for arbitrary 
$(\hat{x}, \hat{\xi})\in \mc{F}$, so $f_2-p$ vanishes on $\mc{F}$.
\end{proof}
Note that $\mc{F} = X\times S$ when the second-stage problem of 
\reff{eq:2stageSP} is unconstrained. We then have the following result as a special case of Theorem~\ref{thm:poly_recourse}.
\begin{corollary}\label{thm:uc_poly_recourse}
Suppose that the second-stage problem of \reff{eq:2stageSP} is unconstrained. If there exists 
$q\in\pre{\tilde{g}}$ such that $F - q\in\re[x,\xi]$, and the set $\mc{V}_q(x,\xi)$
is nonempty for every $x\in X$ and $\xi\in S$, then the recourse function of \reff{eq:2stageSP}
satisfies $f_2(x,\xi) = F(x,y,\xi)-q(x,y,\xi)$ for any $(x,\xi)\in X\times S$.
\end{corollary}
We give an example of constrained second-stage optimization that has a polynomial 
recourse function.
\begin{example}
Given  $x\in\re^1$  and  $\xi\in\re^1$, 
consider the second-stage optimization problem 
\[
\left\{\begin{array}{rl}
f_2(x,\xi) =  \min\limits_{y\in\re^2} & F(x,y,\xi) = x^2y_1-xy_2^2\\
\st & y_1-x\ge 0,\, y_2\ge 0,\; x+\xi-y_1-y_2\ge 0.
\end{array}\right.
\] Assume $S = X = [0,1]$ are determined by
$(\xi,\, 1-\xi)\ge 0$ and $(x,1-x)\ge 0$, respectively. Then $\mc{F} = X\times S$.
Denote the tuple of constraining polynomials
\[ \tilde{g}(x,y,\xi) \,=\, (\xi,\,1-\xi,\, x,\, 1-x,\, y_1-x,\, y_2,\, x+\xi-y_1-y_2). \]
Let $q\in \pre{\tilde{g}}$ be given as
\[
q(x,y,\xi) \,=\, x^2(y_1-x) + xy_2(y_1-x) + x\xi(y_1-x) + xy_2(x+\xi-y_1-y_2) + x\xi(x+\xi-y_1-y_2).
\]
For every $(x,\xi)\in \mc{F}$, the set $\mc{V}_{q}(x,\xi)$ in \eqref{eq:Vq} is not empty since 
it always contains $y = (y_1,y_2) = (x,\xi)$. In addition, it is easy to compute that 
\[
F(x,y,\xi)-q(x,y,\xi) \,=\, x^3-x\xi^2\,\in\,\re[x,\xi].
\]
Then by Theorem~\ref{thm:poly_recourse}, the recourse function of this problem
is $f_2(x,\xi) = x^3-x\xi^2$.
\end{example}

\section{Algorithms for Solving Two-Stage SPs}
\label{sc:paa}

In this section, we introduce a polynomial approximation framework to solve 
the two-stage SP \reff{eq:2stageSP}, which is restated here for convenience:
\[
\min\limits_{x\in X}\quad f(x)\, \coloneqq\,  
f_1(x)+\mathbb{E}_{\mu} [f_2(x,\xi)].
\]
Our algorithm has two phases. First, we compute a polynomial lower approximating function $p(x,\xi)$ 
for the recourse function $f_2(x,\xi)$, leveraging the optimization problem \reff{eq:maxim:int} or \reff{eq:maxim:int:dd}. 
Subsequently, we approximate the first-stage problem \reff{eq:2stageSP} via  
\begin{equation}\label{eq:alg:lb}
\min_{x\in X}\quad \tilde{f}(x) \, \coloneqq\, 
f_1(x)+\mathbb{E}_{\mu}[p(x,\xi)].
\end{equation}
The optimal value of the above problem yields a lower bound for the optimal value of the original two-stage SP.
If $\tilde{x}$ is a global optimizer of \reff{eq:alg:lb}, and given $f(x)-\tilde{f}(x)\ge 0$
for every $x\in X$, it follows that 
\[
\tilde{f}(\tilde{x})\quad \le\quad \min\limits_{x\in X}\,f(x)
\quad\le\quad f(\tilde{x}).
\]
In the case where $\tilde{f}(\tilde{x}) = f(\tilde{x})$, we can confirm the global optimality of $\tilde{x}$ for the
original two-stage SP. Otherwise, we can use $\tilde{x}$ to refine the probability
measures $\nu$ in \reff{eq:maxim:int} or $\nu_i$ in \reff{eq:maxim:int:dd}, facilitating the determination of a subsequent lower approximating function and an improved objective value of \eqref{eq:alg:lb}.
Since \reff{eq:maxim:int} seeks to minimize $\int_{\mc{F}} |f_2(x,\xi)-p(x,\xi)|{\tt d}\nu$,
we suggest updating
\[
\nu\,\coloneqq\, \alpha \nu + (1-\alpha)(\delta_{\tilde{x}}\times \mu)\quad 
\mbox{with a small $\alpha\in (0,1)$,}
\]
where $\delta_{\tilde{x}}$ denotes the Dirac measure supported at $\tilde{x}$.
This strategy ensures that the newly computed lower bound functions more accurately approximate the true recourse function in the 
neighborhood of previous candidate solutions. A similar strategy is recommended to update
$\nu_i\coloneqq \alpha\nu_i+(1-\alpha)\delta_{\tilde{x}}$ in \reff{eq:maxim:int}.
Moreover, it is desirable to ensure that the optimal objective values computed from the approximating problem \eqref{eq:alg:lb} exhibit an increasing trend along the iterations. Therefore, in the next iteration, we add the following  constraint to compute a new lower bound function:
\begin{equation}\label{eq:f1+E-f>0}
f_1(\tilde{x})+ \mathbb{E}_{\mu}[p(\tilde{x},\xi)] - \tilde{f}(\tilde{x}) \,\ge\,0.
\end{equation}
This iterative process is repeated  until the difference between  the computed largest lower bound and the
 smallest upper bound for the optimal value of \reff{eq:2stageSP} is sufficiently small.
We summarize the entire procedure in the following algorithm.
\begin{algorithm}
\label{def:alg}
For the two-stage SP \reff{eq:2stageSP},  proceed as follows:
\begin{description}
\item[Step 0  (Initialization):] 
Let $\alpha\in(0,1)$ be a given scalar, $\epsilon\ge 0$ be a given tolerance and $\nu$ be a 
probability measure supported on $\mc{F}$. Select the degree of polynomial lower approximating functions.
Set $v^+ \coloneqq +\infty$ and $v^- \coloneqq -\infty$.  

\item[Step 1 (Lower Approximating Functions Generation):]
Solve the optimization problem \reff{eq:maxim:int} to get a polynomial lower approximating function $p(x,\xi)$
at a given degree.

\item[Step 2 (Lower and Upper Bounds Update):]
Let $\tilde{f}(x)\coloneqq f_1(x)+\mathbb{E}_{\mu}[p(x,\xi)]$. Solve the optimization problem 
\reff{eq:alg:lb} for an optimal solution $\tilde{x}$. Update $v^- \coloneqq \max\{v^-,\tilde{f}(\tilde{x})\}$.
If $v^+ > f(\tilde{x})$, write $\tilde{x}^* \coloneqq \tilde{x}$ and update $v^+\coloneqq f(\tilde{x})$.
 
\item[Step 3 (Termination Check):]	 
If $v^+-v^-\le \epsilon$, let $\tilde{f}^*\coloneqq v^-$. Stop and output $\tilde{x}^*$ and $\tilde{f}^*$ as
an (approximate)  optimal solution  and an optimal value of \reff{eq:2stageSP}, respectively.
Otherwise, add the new constraint \reff{eq:f1+E-f>0} in \reff{eq:maxim:int} and update 
$\nu\,\coloneqq\, \alpha \nu + (1-\alpha)(\delta_{\hat{x}}\times\mu)$. Then go back to Step~1.
\end{description}
\end{algorithm}
We make some remarks for the above algorithm.

In Step~0, the degree of polynomial lower bound functions is predetermined for the sake of computational feasibility. When $\mc{F}$ is a simple set such as boxes, simplex or balls,
the probability measure $\nu$ can be conveniently chosen to be the uniform distribution. 
In cases where $\mc{F}$ is compact yet possesses complex geometrical characteristics, 
we often construct $\nu$ as a finitely atomic measure derived from sampling procedures.
For instance, if $\mc{F}\subseteq [-R,R]^{n_1\times n_0}$ for  for some sufficiently large $R>0$, 
we  would first generate samples following distribution supported on $[-R,R]^{n_1\times n_0}$,
and then select those in $\mc{F}$ as the finite support of $\nu$.

In Step~1, the optimization problem \reff{eq:maxim:int} is a linear conic optimization problem
with a nonnegative polynomial cone. This problem can be relaxed to a hierarchy of linear semidefinite programs. 
Under the archimedean assumption, we can solve for a sequence of optimizing polynomials 
of \reff{eq:maxim:int} from these relaxations. 
In Step~2, \reff{eq:alg:lb} is a deterministic polynomial optimization problem, 
which can be solved globally by Moment-SOS relaxations.
Detailed discussions on Moment-SOS relaxations are given in Section~\ref{sc:lco}.

In Steps~2 and 3, one needs to compute the expectation $\mathbb{E}_{\mu}[\cdot ]$ to 
evaluate $f(\tilde{x})$, which can be estimated via the sample average when $\xi$ follows a continuous distribution. 
The implementation of such methods is introduced in Section~\ref{sc:numexp}.
It is clear that $v^+$ is an upper bound and $v^-$ is a lower bound for the optimal value of 
\reff{eq:2stageSP}. Notice that when the algorithm terminates, the output solution $\tilde{x}^*$ 
satisfies $f(\tilde{x}^*) = v^-$, but $\tilde{x^*}$ may not be the optimizer $\tilde{x}$ 
computed in the last iterate.
\begin{prop}\label{prop:alg1}
Suppose that  $f^*$ is the global optimal value of \reff{eq:2stageSP}.
If Algorithm~\ref{def:alg} terminates with an output pair $(\tilde{x}^*, \tilde{f}^*)$,
then 
\[\tilde{f}^* \,\le\, f^*\,\le\,\tilde{f}^*+\epsilon,\quad 
f(\tilde{x}^*)-\epsilon\,\le\,f^* \,\le\, f(\tilde{x}^*).\]
For the special case where $\epsilon = 0$, we have $f^*=\tilde{f}^*$ and $\tilde{x}^*$ is 
a global optimal solution of \reff{eq:2stageSP}.
\end{prop}
\begin{proof}
By given conditions, $\tilde{x}^*$ is the optimizer of \reff{eq:alg:lb} at  some iterate $t$.
Let $\tilde{f}_t(x)$ denote the objective function of \reff{eq:alg:lb} at the same iterate.
Since $f(x)-\tilde{f}_t(x)\ge 0$ for every $x\in X$, we have
\[ \tilde{f}^*\,=\, \min\limits_{x\in X} \tilde{f}_t(x) \,\le\,  \min\limits_{x\in X} f(x) 
\,=\,f^*\,\le\, f(\tilde{x}^*). \]
For Algorithm~\ref{def:alg} to terminate, we must have $f(\tilde{x}^*) - \tilde{f}^*\le \epsilon$,
thus $f(\tilde{x}^*)-\epsilon\le f^*\le \tilde{f}^*+\epsilon$.
For the special case where $\epsilon = 0$, we have $f(\tilde{x}^*) = f^* =\tilde{f}^*$,
so $\tilde{x}^*$ is a global optimizer of \reff{eq:2stageSP}.
\end{proof}

\subsection{The case where $\xi$ has a finite support}
In this subsection, we
consider the special case where $\xi$ possesses a finite support $S=\{ \xi^{(1)}, \ldots, \xi^{(r)}\}$.
Suppose
\begin{equation}\label{eq:mudis}
\mu \,=\, \lambda_1\delta_{\xi^{(1)}} + \cdots +\lambda_r\delta_{\xi^{(r)}}, 
\end{equation}
where each $\lambda_i>0$ and $\lambda_1+\cdots+\lambda_r=1$. Under this structure, 
we can construct the lower bound function of $p(x,\xi)$ as in \reff{eq:lbf:dis}:
\[p(x,\xi)\,\coloneqq\, \lambda_1p_1(x)+\cdots  +\lambda_rp_r(x), \]
where each $p_i(x)$ is solved from the linear conic optimization \reff{eq:maxim:int:dd}.
Then we propose the following variant  of Algorithm~\ref{def:alg}.
\begin{algorithm}
\label{def:alg:dis}
For the two-stage SP \reff{eq:2stageSP} with $\nu$ given in \reff{eq:mudis}, proceed as follows:
\begin{description}
\item[Step 0  (Initialization):] 
Let $\alpha\in(0,1)$ be a given scalar and $\epsilon\ge 0$ be a given tolerance.
Choose the degree of lower bound functions. Set $T \coloneqq [r]$ and  $v^+ \coloneqq +\infty$.
For every $i\in T$, fix a probability measure $\nu_i$ supported on $\mc{F}_i$,
and let $v_i^-\coloneqq -\infty$.

\item[Step 1 (Lower Approximating Functions Generation):]
For every $i\in T$, solve the optimization problem \reff{eq:maxim:int:dd} for a polynomial
lower approximating function $p_i(x)$ of the given degree.

\item[Step 2 (Lower and Upper Bounds Update):]
Let $\tilde{f}(x)\coloneqq f_1(x)+\mathbb{E}_{\mu}[p(x,\xi)]$ with $p(x,\xi)$ as in \reff{eq:lbf:dis}. 
Solve the optimization problem \reff{eq:alg:lb}  to get an optimal solution $\tilde{x}$. For each $i\in T$, 
update $v_i^-\coloneqq \max\{v_i^-, p_i(\tilde{x})\}$. If $v^+>f(\tilde{x})$, write 
$\tilde{x}^*\coloneqq \tilde{x}$ and update $v^+\coloneqq f(\tilde{x})$.
 
\item[Step 3 (Termination Check):]	 
Update $T\,\coloneqq \, \{i\in [r]: f_2(\tilde{x}^*, \xi^{(i)}) - v_i^->\epsilon\}$.
If $T=\emptyset$, let $\tilde{f}^* \coloneqq \lambda_1v_1^- + \cdots +\lambda_rv_r^-$.
Stop and output $\tilde{x}^*$ and $\tilde{f}^*$ as
an (approximate)  optimal solution  and an optimal value of \reff{eq:2stageSP}, respectively. Otherwise, add the new constraint 
$p_i(\tilde{x}) \ge v_i^-$ to \reff{eq:maxim:int:dd} 
and update $\nu_i := \alpha\nu_i + (1-\alpha) \delta_{\tilde{x}}$ for all $i\in T$.
\end{description}
\end{algorithm}

The above framework has a major difference from Algorithm~\ref{def:alg}. In each iteration,
Algorithm~\ref{def:alg}  computes a single lower bound function $p(x,\xi)$, whereas
Algorithm~\ref{def:alg:dis} computes $|S|$ many polynomials $p_i(x)$ each time. When $|S|$ is small
and $\xi$ is of large dimension, Algorithm~\ref{def:alg:dis} can be more computationally 
efficient than Algorithm~\ref{def:alg}. By setting $\deg(p(x,\xi)) = \deg(p_i(x))$, the problem 
 \reff{eq:maxim:int:dd} has much fewer variables than \reff{eq:maxim:int}, which allows for faster and more robust computation of each individual 
optimization problem. When $S$ contains infinitely many elements,
 Algorithm~\ref{def:alg:dis} may
still be applied using sampling methods, although the  number of lower bound functions computed in each
iteration increases linearly with the size of the samples.

Similar to Algorithms~\ref{def:alg}, 
all optimization problems in Algorithm~\ref{def:alg:dis} can be efficiently solved using Moment-SOS relaxations. 
Additionally, Algorithm~\ref{def:alg:dis} shares similar convergence properties as described in Proposition~\ref{prop:alg1}.

\begin{prop}\label{prop:alg:dis}
Suppose that  $f^*$ is the global optimal value of \reff{eq:2stageSP}, where $\xi$ possesses a finite
support $S = \{\xi^{(1)}, \ldots, \xi^{(r)}\}$.
If Algorithm~\ref{def:alg:dis} terminates with an output pair $(\tilde{x}^*, \tilde{f}^*)$,
then 
\[\tilde{f}^* \,\le\, f^*\,\le\,\tilde{f}^*+\epsilon,\quad 
f(\tilde{x}^*)-\epsilon\,\le\,f^* \,\le\, f(\tilde{x}^*).\]
For the special case where $\epsilon = 0$, we have $f^*=\tilde{f}^*$ and $\tilde{x}^*$ is 
a global optimal solution of \reff{eq:2stageSP}.
\end{prop}
\begin{proof}
It is evident that $f(\tilde{x}^*)\ge f^*$. Recall that $\tilde{f}^* \coloneqq \lambda_1v_1^- + \cdots +\lambda_rv_r^-$.
Since each $v_i^-$ provides a lower bound for $f_2(x,\xi^{(i)})$ over all $x\in X$, it follows that $f^*\ge \tilde{f}^*$. Upon the termination of  Algorithm~\ref{def:alg:dis}, the condition
$f_2(\tilde{x}, \xi^{(i)})-v_i^-\le \epsilon$ must hold for each $i\in[r]$. Consequently, we have
\[
f(\tilde{x}^*) - \tilde{f}^* 
\,=\, \sum\limits_{i=1}^r \lambda_i (f_2(x,\xi^{(i)}) - v_i^-)
\,\le\, \epsilon \Big(\sum\limits_{i=1}^r \lambda_i\Big) \,=\,\epsilon.
\]
Employing similar arguments to that in Proposition~\ref{prop:alg1}, one can derive all stated results.
\end{proof}

\section{Moment-SOS Relaxations}\label{sc:lco}
In this section, we introduce Moment-SOS relaxation methods for solving linear conic optimization and
polynomial optimization problems in Algorithms~\ref{def:alg} and \ref{def:alg:dis}.
For the two-stage SP \reff{eq:2stageSP}, denote tuples of constraining polynomials
\begin{equation}\label{eq:g1g2}
g_0(\xi) \coloneqq (g_{0,i}(\xi))_{i\in\mc{I}_0},\quad
g_1(x)\,\coloneqq \, (g_{1,i}(x))_{i\in\mc{I}_1} ,\quad 
g_{2}(x,y,\xi) \,\coloneqq \, (g_{2,i}(x,y,\xi))_{i\in\mc{I}_2}.
\end{equation}

\subsection{Relaxations of problem \reff{eq:maxim:int}}
The linear conic optimization problem \reff{eq:maxim:int} is
\[
\left\{\begin{array}{cl}
\max\limits_{p \in \re[x,\xi]} & \int_{\mc{F}} p(x,\xi) {\tt d}{\nu}\\
\st & F(x,y,\xi)-p(x,\xi)\in\mathscr{P}(K),
\end{array}
\right.
\]
where $\nu$ is a given measure and $K$ is a semialgebraic set determined by
\begin{equation}\label{eq:Kalg} 
g_0(\xi)\ge 0,\quad  g_1(x)\ge 0,\quad  g_2(x,y,\xi)\ge 0. 
\end{equation}
The nonnegative polynomial cone $\mathscr{P}(K)$ typically does not have a 
convenient expression in computations. Note  that $g_0,g_1,g_2$ can all be viewed 
as tuples of polynomials in $(x,y,\xi)$. Denote the quadratic module as
\[ \qmod{g_0,g_1,g_2}\,\coloneqq \, \qmod{g_0}+\qmod{g_1}+\qmod{g_2}, \]
where (recall $\Sigma[x,y,z]$ is the SOS polynomial cone)
\[ \qmod{g_j} = \sum_{i\in \mc{I}_j} \big( g_{j,i}(x,y,\xi)\cdot \Sigma[x,y,z] \big),\quad j = 1,2,3. \]
Let $\qmod{g_0,g_1,g_2}_{2k}\coloneqq \qmod{g_0,g_1,g_2}\cap \re[x,y,\xi]_{2k}$ be 
the $k$th order truncation. It can be explicitly expressed with semidefinite constraints.
We can use these truncated quadratic modules to approximate $\mathscr{P}(K)$.
Indeed, for a given degree $d$, if $\qmod{g_0,g_1,g_2}$ is archimedean, it holds that 
\begin{equation}\label{eq:PdK}
\mbox{int}\Big(\mathscr{P}_d(K)\Big) \,=\, 
\bigcap\limits_{k\ge \lceil d/2\rceil} \Big(\qmod{g_0,g_1,g_2}_{2k}\cap \re[x,y,\xi]_d\Big).
\end{equation}
Then we can construct a hierarchy of semidefinite relaxations of \reff{eq:maxim:int}.
For $k$ with $2k\ge \deg(F)$, the $k$th order SOS relaxation of \reff{eq:maxim:int}   is
\begin{equation}\label{eq:polysos:k}
\left\{\begin{array}{cl}
\max\limits_{p\in\re[x,\xi]} & \int_{\mc{F}} p(x,\xi){\tt d}\nu\\
\st & F(x,y,\xi) - p(x,\xi)\in \qmod{g_0,g_1,g_2}_{2k}.
\end{array}
\right.
\end{equation}
Its dual problem is called the $k$th order moment relaxation of \reff{eq:maxim:int}.
The problem \reff{eq:polysos:k} is a linear conic optimization problem, where the coefficient vector of $p(x,\xi)$
is the decision vector. For $p(x,\xi)$ to be feasible for \reff{eq:polysos:k}, its total degree must be 
smaller than or equal to $2k$. Since $\qmod{g_0,g_1,g_2}_{2k}$ can be expressed by semidefinite 
constraints, the optimization problem \reff{eq:polysos:k} can be solved efficiently by interior point methods.
\begin{theorem}\label{thm:asymconv}
Suppose $\qmod{g_0,g_1,g_2}$ is archimedean and $f_2(x,\xi)$ is continuous on $\mc{F}$.
For a given probability measure $\nu$, problem \reff{eq:polysos:k} is solvable with an optimal solution 
$p^{(k)}(x,\xi)$ when $k$ is large enough, and 
\[
\int_{\mc{F}} |f_2(x,\xi)-p^{(k)}(x,\xi)| {\tt d}{\nu}\rightarrow 0
\quad\mbox{as}\quad k\rightarrow\infty.
\]
\end{theorem}
\begin{proof} 
Under the archimedean condition, $K$ in \reff{eq:Kalg} is compact and $\qmod{g_0,g_1,g_2}_{2k}$
is closed for every $k$. Then $\mc{F}$ is also compact as a projection of $K$ onto the $(x,\xi)$ space.
Since $f_2(x,\xi)$ is continuous on $\mc{F}$, by Theorem~\ref{thm:gam_equal},  
for every $\varepsilon>0$, there exists a polynomial $p(x,\xi)$ that is feasible for 
\reff{eq:maxim:int} and satisfies $\int_{\mc{F}} |f_2(x,\xi)-p(x,\xi)| {\tt d}{\nu}\le \varepsilon$.
Then 
\[ F(x,y,\xi)- (p(x,\xi)-\varepsilon)\, \ge\, \varepsilon \,>\,0,\quad \forall (x,y,\xi)\in K. \] 
By Putinar's Positivstellensatz, $F(x,y,\xi) - (p(x,\xi)-\varepsilon)\in \qmod{g_0,g_1,g_2}$.
So there exists $k_{\varepsilon}\in\N$ that is sufficiently large such that the polynomial 
$p(x,\xi)-\varepsilon$ is feasible for \reff{eq:polysos:k} at the $k_{\varepsilon}$th relaxation. 
At the $k_{\varepsilon}$th relaxation, \reff{eq:polysos:k} is bounded from above and 
has a nonempty closed feasible set,  so it is solvable with an optimizer $p^{(k_{\varepsilon})}(x,\xi)$.
Then we have 
\[
\int_{\mc{F}} |f_2(x,\xi) - p^{(k_{\varepsilon})}(x,\xi)|{\tt d}\nu\,\le\,
\int_{\mc{F}} |f_2(x,\xi) - (p(x,\xi)-\varepsilon)| {\tt d}{\nu}\,\le\, 2\varepsilon.
\]
Since $\qmod{g_0,g_1,g_2}_{2k}\subseteq \qmod{g_0,g_1,g_2}_{2k+2}$ for every $k$,
the optimal value of \reff{eq:polysos:k} increases monotonically as the relaxation order grows.
In other words, $k_{\varepsilon}\rightarrow \infty$ as $\varepsilon\rightarrow 0$.
So the conclusion holds.
\end{proof}
For the special case that $f_2(x,\xi)$ is a polynomial and $F-f_2\in \qmod{g_0,g_1,g_2}$,
the true recourse function is an optimizer of \reff{eq:polysos:k} when $k$ is big enough.
Since $p(x,\xi)$ has two kinds of variables $x$ and $\xi$, one can also use a pair of degrees 
as the relaxation order. Denote 
\begin{equation}\label{eq:dx}
d_1 \,\coloneqq\ \max \left\{\, \deg_x(F),\quad  \deg(g_1),
\quad \deg_x(g_2)  \,\right\},
\end{equation}
\begin{equation}\label{eq:dxi}
d_2 \,\coloneqq\, \max \left\{ \deg_{\xi}(F),\quad \deg(g_0),
\quad \deg_{\xi}(g_2) \,\right\}.
\end{equation}
Let $\mathbf{k} = (k_1,k_2,k)$ such that $k_1\ge d_1$, $k_2\ge d_2$, $k =\max\{\lceil (k_1+k_2)/2\rceil, \lceil \deg(F)/2\rceil\}$.
The $\mathbf{k}$th order SOS relaxation of \reff{eq:maxim:int} is
\begin{equation}
\label{eq:polysos:veck}
\left\{
\begin{array}{cl} 
\max & \int_{\mc{F}} p(x,\xi){\tt d} \nu\\[0.05in]
\st & F(x,y,\xi) - p(x,\xi) \in \qmod{g_1,g_2,g_3}_{2k}.\\[0.05in]
& p(x,\xi) \in \re[x,\xi]_{k_1,k_2}.\\
\end{array}
\right.
\end{equation}
In the above, $\re[x,\xi]_{k_1,k_2}$ is the set of real polynomials with partial degrees in $x$ 
no more than $k_1$ and partial degree in $\xi$ no more than $k_2$.
Let $v^{(k)}$ denote the optimal value of \reff{eq:polysos:k} and let $v^{(\mathbf{k})}$ 
denote the optimal value of \reff{eq:polysos:veck}.
We have $v^{(k)} \,\ge \, v^{(\mathbf{k})}$ for every $\mathbf{k} = (k_1,k_2,k)$
such that $k = \lceil(k_1+k_2)/2\rceil$.
\begin{corollary}
Suppose $\qmod{g_1,g_2,g_3}$ is archimedean and $f_2(x,\xi)$ is continuous on $\mc{F}$.
For a given measure $\nu$, problem \reff{eq:polysos:veck} is solvable with an optimal solution 
$p^{(\mathbf{k})}(x,\xi)$ with $\mathbf{k} = (k_1,k_2,k)$ when $\min(\mathbf{k})$ is large enough,
and
\[ \int_{\mc{F}} |f_2(x,\xi)-p^{(\mathbf{k})}(x,\xi)| {\tt d}{\nu}\rightarrow 0
\quad\mbox{as}\quad \min(\mathbf{k})\rightarrow\infty. \]
\end{corollary}
\begin{proof}
This result is implied by Theorem~\ref{thm:asymconv}.
\end{proof}
We remark that the relaxation \reff{eq:polysos:veck} is more flexible than \reff{eq:polysos:k} in computations.
By adjusting the degrees of $x$ and $\xi$ separately, we can construct lower approximating functions $p(x,\xi)$ 
with different focus on the decision variables and the random variables. 
In addition, for a fixed $k = \lceil (k_1+k_2)/2\rceil$, problem \reff{eq:polysos:veck} has fewer variables than \reff{eq:polysos:k}, while the computed lower approximating functions may still be very efficient.
Here is such an example.
\begin{example}\label{ex:uni_uniform}
Consider a two-stage SP as in \reff{eq:2stageSP} with $x,y,\xi\in\re$, 
$f_1(x) = 0$, $\mu\sim \mc{U}(S)$, and 
\[
X = \{x\in\re: 1-x^2\ge 0\},\quad S = \{\xi\in\re: \xi(1-\xi)\ge 0 \},
\]
where $\mc{U}(S)$ denotes the uniform distribution on $S$.
The second-stage problem is given as
\[
\left\{\begin{array}{rl}
 f_2(x,\xi) = \min\limits_{y\in\re} & (x+\xi)y^3-\xi y^2+xy\\
\st & \xi^2-(y-x)^2\ge 0.
\end{array}
\right.
\] 
Clearly, the second-stage problem is feasible for every $x\in X$ and $\xi\in S$, so $\mc{F} = X\times S$.
Select $\nu$ to be the uniform probability measure supported on $\mc{F}$.
We solve lower approximating functions from the SOS relaxations \reff{eq:polysos:k} 
with different relaxation orders $\mathbf{k} = (k_1,k_2,k)$. The resulting polynomials are listed in the following table. 
\[
\begin{array}{|c|c|l|}
\hline
(k_1,k_2) & k & p^{(\mathbf{k})}(x,\xi)\\
\hline
(1,\,2) & 2 & -0.3426+0.4788x+2.2407\xi-3.1747x\xi-4.0833\xi^2+4.8810x\xi^2\\
\hline
(1,\,3) & 2 & -0.0042+0.0565x-0.3476\xi-1.1198x\xi+1.7471\xi^2+2.3027
x\xi^2\\
& &-3.5887\xi^3+0.9257x\xi^3\\
\hline
(2,\,2) & 2 & -0.4450+0.5490x+0.8802x^2+2.4376\xi-3.2883x\xi-0.4785x^2\xi\\
& &-4.1466\xi^2+ 4.9806x\xi^2 -0.5446x^2\xi^2\\
\hline
(2,\,3) & 3 & -0.0903-0.0036x+1.4816x^2-0.0754\xi+0.4759x\xi-3.9125x^2\xi\\&  &+1.0345\xi^2 -2.7192x\xi^2+5.9542x^2\xi^2-3.0738\xi^3+4.4429x\xi^3\\
& &-3.5727x^2\xi^3\\
\hline
\end{array}
\]
Then we compute $f^{(\mathbf{k})}(x)\coloneqq \mathbb{E}_{\mu}[p^{(\mathbf{k})}(x,\xi)]$ 
for each above $\mathbf{k}$ and plot them with the true expected recourse function 
$f(x) = \mathbb{E}_{\mu}[f_2(x,\xi)]$ in Figure~\ref{fig:uni_uniform}.
Specifically, the function $f^{(1,2,2)}$ is plotted in the dashed line, 
the function $f^{(1,3,2)}$ is plotted in the dotted line,
the function $f^{(2,2,2)}$ is plotted in the dash-dotted line,
the function $f^{(2,3,3)}$ is plotted in the plus sign line,
and the expected recourse $f$ is plotted in the solid line.
In addition, we plot global minimizers of all these $f^{(\mathbf{k})}(x)$ on $X$ in blue dots.
\begin{figure}[htb!]
\centering
\includegraphics[width=0.4\textwidth]{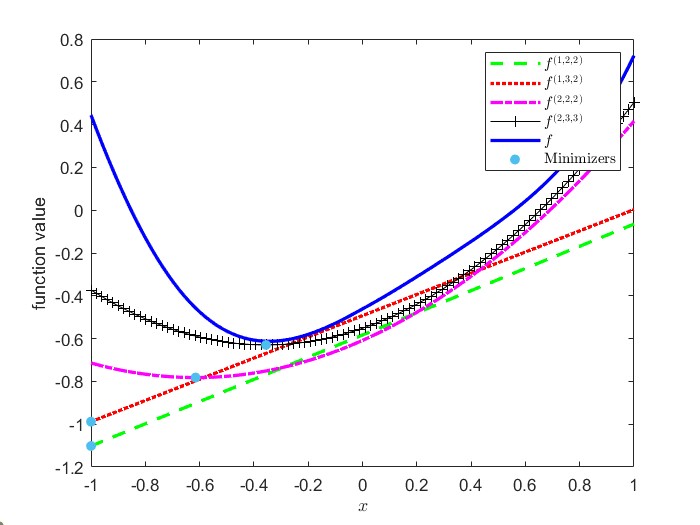}
\caption{Compute $\mathbb{E}_{\mu}[p^{(\mathbf{k})}(x,\xi)]$ and $\mathbb{E}[f_2(x,\xi)]$ 
for Example~\ref{ex:uni_uniform}, where
dashed line is $f^{(1,2,2)}$, dotted line is $f^{(1,3,2)}$, dash-dotted line is $f^{(2,2,2)}$,
plus sign line is $f^{(2,3,3)}$, solid line is $f$, and big dots are minimizers.}
\label{fig:uni_uniform}
\end{figure}
Clearly, the global minimum of $\mathbb{E}_{\mu}[p^{(\mathbf{k})}(x,\xi)]$ on $X$ increases as 
the relaxation order increases.
Denote by $f_{min}^{(\mathbf{k})}$ and $x^{\mathbf{(k)}}$ the global minimum and 
minimizer of \reff{eq:alg:lb}.
We report the computational results in the following table.
\[
\begin{array}{|c|c|c|c|c|}
\hline
\mathbf{k} & (1,2,2) & (1,3,2) & (2,2,2) & (2,3,3) \\
\hline
f_{min}^{(\mathbf{k})} & -1.1018 & -0.9883 & -0.7821 & -0.6296 \\
\hline
x^{(\mathbf{k})} & -1.0000 & -1.0000 & -0.6149 & -0.3555 \\
\hline
\end{array}
\]
From Figure~\ref{fig:uni_uniform}, one can observe that when $\mathbf{k} = (2,3,3)$, 
$x^{(\mathbf{k})} = -0.3555$ is close to the global 
optimizer of the two-stage SP.
By sample average approximations, we compute 
\[
f( -0.3555) \,\approx\, \frac{1}{100}\sum\limits_{i = 1}^{100} f_2( -0.3555,\, 0.01\cdot i)  = -0.6042,
\]
which is close to $f^{(\mathbf{k})}(-0.3555) = -0.6296$.
One can further improve the approximation quality by increasing the relaxation order.
\end{example}

\subsection{Relaxations of problem \reff{eq:maxim:int:dd}}
For ease of reference, we repeat the optimization problem \reff{eq:maxim:int:dd} below:
\[
\left\{\begin{array}{cl}
\max\limits_{p_i\in\re[x] } & \int_{\mc{F}_i} p_i(x) {\tt d}\nu_i\\
\st & F(x,y,\xi^{(i)}) - p_i(x)\in \mathscr{P}(K_i)^{x,y},
\end{array}
\right.
\]
where $\nu_i$ is a given probability measure supported on $\mc{F}_i$,
and $K_i$ is a semialgebraic set determined by
\begin{equation}\label{eq:Kipoly} 
g_1(x)\ge 0,\quad  g_2(x,y,\xi^{(i)})\ge 0.
\end{equation}
The functions $g_1,g_2(\bullet, \xi^{(i)})$ can be viewed as polynomial tuples in $(x,y)$.
Denote the quadratic module
\[
\qmod{g_1, g_2(\bullet, \xi^{(i)})}^{x,y}\,\coloneqq\, \qmod{g_1}^{x,y}
 + \qmod{g_2(\bullet, \xi^{(i)})}^{x,y}
\]
as a subset in $\re[x,y]$, where 
\[
\qmod{g_1}^{x,y}\,\coloneqq\,\sum\limits_{i\in\mc{I}_1} g_{1,i}(x)\cdot \Sigma[x,y],\quad
\qmod{g_2(\bullet, \xi^{(i)})}^{x,y}\,\coloneqq\,
\sum\limits_{j\in\mc{I}_2} g_{2,j}(x,y,\xi^{(i)})\cdot \Sigma[x,y].
\]
Let $k\ge \max\{\lceil \deg(F)/2\rceil, \lceil d_1/2\rceil\}$. 
The $k$th order SOS relaxation of \reff{eq:maxim:int:dd} is 
\begin{equation}\label{eq:kth:dis}
\left\{\begin{array}{cl}
\max\limits_{p_i\in \re[x]} & \int_{\mc{F}_i} p_i(x){\tt d}{\nu_i}\\
\st & F(x,y,\xi^{(i)}) - p_i(x)\in \qmod{g_1, g_2(\bullet,\xi^{(i)})}_{2k}^{x,y},
\end{array}
\right.
\end{equation}
where  $\qmod{g_1,g_2(\bullet,\xi^{(i)})}_{2k}^{x,y}$ denotes the $k$th order truncation 
of $\qmod{g_1,g_2(\bullet, \xi^{(i)})}^{x,y}$.
\begin{theorem}
Suppose $\qmod{g_1,g_2(\bullet,\xi^{(i)})}^{x,y}$ is archimedean and 
$f_2(x,\xi^{(i)})$ is continuous on $\mc{F}_i$. For a given measure $\nu_i$, problem 
\reff{eq:kth:dis} is solvable with an optimal solution $p_i^{(k)}(x)$ when $k$ is large enough, and
\[
\int_{\mc{F}_i} |f_2(x,\xi^{(i)})-p_i^{(k)}(x)| {\tt d}{\nu_i}\rightarrow 0
\quad\mbox{as}\quad k\rightarrow \infty.
\]
\end{theorem}
\begin{proof} 
Under the archimedean condition,  $K_i$ in \reff{eq:Kipoly} is compact, thus $\mc{F}_i$ is compact.
Assume $f_2(x,\xi^{(i)})$ is continuous on $\mc{F}_i$.
By Theorem~\ref{thm:gam_equal_2}, for every $\varepsilon>0$, there exists a polynomial 
$p_i(x)$ that is feasible for \reff{eq:maxim:int:dd} and 
satisfies $\int_{\mc{F}_i} |f_2(x,\xi^{(i)})-p_i(x)| {\tt d}{\nu_i}\le \varepsilon$; that is, 
\[ F(x,y,\xi^{(i)})- (p_i(x)-\varepsilon)\, \ge\, \varepsilon \,>\,0,\quad \forall (x,y)\in K_i. \] 
By Putinar's Positivstellensatz, 
$F(x,y,\xi^{(i)}) - (p_i(x)-\varepsilon)\in \qmod{g_1,g_2(\bullet,\xi^{(i)})}^{x,y}$.
So there exists $k_{\varepsilon}\in\N$ that is sufficiently large such that the polynomial 
$p_i(x)-\varepsilon$ is feasible for \reff{eq:kth:dis} at the $k_{\varepsilon}$th relaxation. 
At the $k_{\varepsilon}$th relaxation, \reff{eq:kth:dis} is bounded from above and 
has a nonempty closed feasible set,  so it is solvable with an optimizer $p_i^{(k_{\varepsilon})}(x)$.
Then we have 
\[
\int_{\mc{F}_i} |f_2(x,\xi^{(i)}) - p_i^{(k_{\varepsilon})}(x)|{\tt d}\nu_i\,\le\,
\int_{\mc{F}_i} |f_2(x,\xi^{(i)}) - (p_i(x)-\varepsilon)| {\tt d}{\nu_i}\,\le\, 2\varepsilon.
\]
Since $\qmod{g_1,g_2(\bullet,\xi^{(i)})}^{x,y}_{2k}\subseteq 
\qmod{g_1,g_2(\bullet,\xi^{(i)})}^{x,y}_{2k+2}$ for every $k$,
the optimal value of \reff{eq:kth:dis} increases monotonically as the relaxation order grows.
In other words, $k_{\varepsilon}\rightarrow \infty$ as $\varepsilon\rightarrow 0$.
So the conclusion holds.
\end{proof}

\begin{example}\label{ex:dis_rel}
Consider the two-stage SP as in \reff{eq:2stageSP} with $x,\xi\in\re$, $f_1(x) = 0$, $y\in\re^2$ and 
\[
S =\{\xi^{(1)}, \xi^{(2)}\} = \{-0.1,0.2\},\quad 
X = \{x\in\re:x(1-x)\ge 0\}.
\]
The second-stage optimization problem  is given as
\begin{equation}
\label{eq:dis_rel}
\left\{\begin{array}{rl}
f_2(x,\xi) = \min\limits_{y\in\re^2} &  x^2y_1+\xi xy_2\\
\st & y_1-\xi \ge 0,\, y_2\ge 0,\\
& x-y_1-y_2 \ge 0.
\end{array}
\right.
\end{equation}
Clearly, $\mc{F}_1 = X = [0,1]$ and $\mc{F}_2 = [0.2,1]$.
Since the second-stage optimization problem is linear in $y$, we can analytically solve
the recourse function at each  realization as
\[
f_2(x,\xi^{(1)}) \,=\, -0.2x^2-0.01x,\quad
f_2(x,\xi^{(2)}) \,=\, 0.2x^2.
\]
Select $\nu_1,\nu_2$ as uniform probability measures supported on $\mc{F}_1,\mc{F}_2$, respectively.
We solve \reff{eq:kth:dis} with initial relaxation order $k=2$.
The computed lower approximating functions are
\[
\begin{array}{l}
p_1^{(2)}(x) \,=\, -0.0004-0.0066x-0.2112x^2+0.0150x^3-0.0069x^4,\\
p_2^{(2)}(x) \,=\, -0.0004+0.0028x+0.1926x^2+0.0084x^3-0.0034x^4.
\end{array}
\]
They provide reasonably good approximations of the true recourse function. 
In fact, we have 
\[
\sup_{x\in X} | f_2(x,\xi^{(1)}) - p_1^{(1)}(x)| \le 4\cdot 10^{-4},\quad
\sup_{x\in X}  | f_2(x,\xi^{(2)}) - p_2^{(2)}(x) | \le 7\cdot 10^{-5}.
\]
Therefore,  for an arbitrary probability measure 
$\mu = \lambda_1\delta_{-0.1}+(1-\lambda_2)\delta_{0.2}$ with $\lambda\in[0,1]$, 
the recourse approximation 
$\tilde{f}(x)\coloneqq \lambda_1\cdot p_1^{(2)}(x) +(1-\lambda)\cdot p_2^{(2)}(x)$ satisfies
\[
\sup_{x\in X} |f(x) - \tilde{f}(x)|\le 4\cdot 10^{-4}.
\]
\end{example}

\subsection{Solving the first-stage problem}
\label{ssc:polyopt}
In this subsection, we discuss how to replace the recourse function $f_2(x,\xi)$ by the approximating polynomial function $p(x,\xi)$ in the  two-stage SP \reff{eq:2stageSP}, and solve the first-stage problem to global optimality. 
Let  $p(x,\xi)$ be a selected polynomial lower approximating function of $f_2(x,\xi)$.
The two-stage SP \reff{eq:2stageSP} can be approximated by the polynomial
optimization problem in \reff{eq:alg:lb}, which takes the form of
\[\left\{\begin{array}{cl}
\min\limits_{x\in \re^{n_1}} & \tilde{f}(x)\coloneqq f_1(x)+\mathbb{E}_{\mu}[p(x,\xi)]\\
\st & g_1(x)\ge 0,
\end{array}\right.
\]  
where $g_1(x) = (g_{1,i}(x))_{i\in\mc{I}_1}$ is the polynomial tuple given as in \reff{eq:g1g2}.
The above problem can be solved globally by Moment-SOS relaxations.
Denote 
\begin{equation}\label{eq:d3}
d_3\,\coloneqq\, \max\big\{ \deg\big(\tilde{f}\big),\quad \deg(g_1) \big\}.
\end{equation}
For $k\in\N$ such that $2k\ge d_3$, the $k$th order SOS relaxation of \reff{eq:alg:lb} is
\begin{equation}\label{eq:max:sos}
\left\{\begin{array}{cl}
\max\limits_{\gamma\in\re} & \gamma\\
\st & \tilde{f}(x)-\gamma\in \qmod{g_1}_{2k}^x,
\end{array}
\right.
\end{equation}
where $\qmod{g}_{2k}^x$  denotes the $k$th order truncation of 
\[
\qmod{g_1}^x\,\coloneqq\, \sum\limits_{i\in\mc{I}_1} g_{1,i}(x)\cdot \Sigma[x].
\]
The dual problem of \reff{eq:max:sos} is the $k$th order moment relaxation
of \reff{eq:alg:lb}, which is
\begin{equation}\label{eq:min:mom}
\left\{
\begin{array}{cl}
\min\limits_{z\in\re^{\N_{2k}^{n_1}}} & \langle \tilde{f}, z\rangle\\
\st & z_0 = 1,\, M_k[z]\succeq 0,\\[0.05in]
&  L_{g_{1,i}}^{(k)}[z]\succeq 0\,(i\in\mc{I}_1).
\end{array}
\right.
\end{equation}
In the above, $M_k[z]$ and each $L_{g_{i,i}}^{(k)}[z]$
 are moment and localizing matrices defined as in \reff{eq:Lq[z]}.
For each $k$, the optimization problems \reff{eq:max:sos}--\reff{eq:min:mom}
are semidefinite programming problems. 
Suppose $\tilde{f}_0$ is the optimal value of \reff{eq:alg:lb} and
$\tilde{f}_k$ is the optimal value of \reff{eq:min:mom} at $k$th relaxation order.
Under the archimedean condition of $\qmod{g_1}^x$, the dual pair
\reff{eq:max:sos}--\reff{eq:min:mom} has the asymptotic convergence (see \cite{Las01})
\[
\tilde{f}_k \le \tilde{f}_{k+1} \le \cdots \le \tilde{f}_0\quad
\mbox{and}\quad \lim_{k\to \infty} \tilde{f}_k = \tilde{f}_0.
\]
Interestingly, the finite convergence, i.e., $\tilde{f}_k = \tilde{f}_0$ for $k$
large enough, holds when $\tilde{f}, g_1$ are given by generic polynomials.
It can be verified by a convenient rank condition called {\it flat truncation} \cite{nie2014optimality}.
Suppose $z^*$ is an optimizer of \reff{eq:min:mom} at the $k$th relaxation.
If there exists $t\in [d_3, k]$ such that
\[
\rank\, M_{t-d_3}[z^*] \, =\, \rank\, M_{t}[z^*],
\]
then \reff{eq:min:mom} is a tight relaxation of \reff{eq:alg:lb}.
In this case, problem \reff{eq:alg:lb} has $\rank\, M_{t}[z^*]$ number of global optimal solutions.
These optimal solutions can be extracted via Schur decompositions \cite{HenLas05}. 
We refer to \cite{Las09book,nie2014optimality,NieBook} for detailed study of polynomial optimization.

\section{Numerical Experiments}
\label{sc:numexp}

In this section, we demonstrate the effectiveness of our method through numerical experiments.
The computations were carried out  in MATLAB R2023a
 on a laptop equipped with an 8th Generation Intel\textregistered Core\texttrademark  i7-12800H CPU and 32 GB RAM.
The computations were implemented with the {\tt MATLAB} software {\tt Yalmip} \cite{Yalmip}, {\tt Mosek} \cite{mosek}
{\tt GloptiPoly 3} \cite{gloptipoly}, and {\tt SeDuMi} \cite{sedumi}.
For clarity, computational results are reported to four decimal places.

In Algorithms~\ref{def:alg} and~\ref{def:alg:dis}, all optimization problems are solved using 
Moment-SOS relaxations. 
For the linear conic optimization problem \reff{eq:maxim:int}, 
we choose a specific relaxation order $\mathbf{k} = (k_1,k_2,k)$ to compute the lower approximating function
$p(x,\xi)$ from problem \reff{eq:polysos:veck}.
For the optimization problem \reff{eq:lbf:dis}, we select a prescribed relaxation order $k$  to determine the
the lower approximating function $p_i(x)$ from \reff{eq:kth:dis}.
The polynomial optimization problem \reff{eq:alg:lb} is globally solved using a hierarchy of semidefinite relaxations, as detailed in \reff{eq:min:mom}.

For the sake of simplicity, we denote the computed lower approximating function for $f(x)$ at the $t$th iteration as $\tilde{f}_t(x)$, with $\tilde{f}_t$ and $\tilde{x}^{(t)}$ representing the global optimal value and the solution obtained from \reff{eq:alg:lb} in the corresponding iteration. We use \texttt{diff} to denote the gap between the upper and lower bounds (i.e., $v^+-v^-$) at each iteration.

First, we consider a synthetic example where the recourse function has an explicit analytical expression.
\begin{example}\label{ex:graphill}
Consider the two-stage SP 
\begin{equation}\label{eq:graphill}
\left\{\begin{array}{cl}
\min\limits_{x\in\re^2} & 2x_1x_2^2-x_1^2+\mathbb{E}_{\mu}[f_2(x,\xi)]\\
\st & 1-x_1^2-x_2^2\ge 0,
\end{array}
\right.
\end{equation}
where $\xi\in\re$ is univariate and $f_2(x,\xi)$ is the optimal value function of the problem
\[
\left\{
\begin{array}{cl}
\min\limits_{y\in\re} & x_2y\\
\st & x_1-2\xi\le y\le x_1+\xi. 
\end{array}
\right.
\]
Since the second-stage problem is linear in $y$ with box constraints, 
one can obtain the following analytical expression of the recourse function:
\[
f_2(x,\xi) = \left\{\begin{array}{ll}
x_2(x_1-2\xi) & \mbox{if $x_2\ge 0$},\\[0.05in]
x_2(x_1+\xi) & \mbox{if $x_2\le 0$}.
\end{array}
\right.
\]
Clearly, $f_2(x,\xi)$ is continuous but nonconvex, and is not a polynomial.
Assume $\mu$ is a probability measure with the support $S  = [0,1]$, moments 
$\mathbb{E}_{\mu}[\xi] = 0.6$ and $
\mathbb{E}_{\mu}[\xi^2]=0.5$.
Then we can find an explicit expression of the overall objective function
\[
f(x) = 2x_1x_2^2-x_1^2+\mathbb{E}_\mu[f_2(x,\xi)] = \left\{\begin{array}{ll}
2x_1x_2^2-x_1^2+x_1x_2-1.2x_2 & \mbox{if $x_2\ge 0$},\\[0.05in]
2x_1x_2^2-x_1^2+x_1x_2+0.6x_2 & \mbox{if $x_2\le 0$}.
\end{array}\right.
\]
One can get the following global optimal solution  and the optimal value of \reff{eq:graphill} by solving two
polynomial optimization problems with Moment-SOS relaxations:
\[
x^* = (-0.6451, 0.7641)^T,\quad f^* = -2.5793.
\]
Now we apply Algorithm~\ref{def:alg} to solve this problem, and compare our results with 
the above true solution. Clearly, $\mc{F} = X\times S$.
Select $\alpha = 0.1$, $\epsilon = 0.001$, and let $\nu$ be the uniform probability measure 
supported on $\mc{F}$. 
For the relaxation order $\mathbf{k} = (2,2,2)$, Algorithm~\ref{def:alg} terminates at the 
initial loop $t= 1$ with the computed objective approximation
\[\tilde{f}_1(x)\,=\,2x_1x_2^2-x_1^2+(-0.6171x_2^2+x_1x_2-0.3000x_2-0.3281).\]
By solving optimization problem \reff{eq:alg:lb},
we obtain the following candidate solution and the corresponding lower bound for the optimal objective value
\[
\tilde{x}^* = (-0.6417,0.7670)^T,\quad
\tilde{f}^* = -2.5801.
\]
Since $f^* = -2.5792$, the gap ${\tt diff} = \tilde{f}^* - \tilde{f}^* = 8.8310\cdot 10^{-4}<0.001$.
Compared to the true optimizer and the optimal value, the computed polynomial lower approximating function, 
even with a low degree,  provides a good approximation.
\end{example}

In the next example,  we show that by increasing the relaxation order, one can
 improve the approximation quality of the polynomial lower approximating functions.
 
\begin{example}\label{ex:x1y2_sim}
Consider the two-stage SP as in \reff{eq:2stageSP} with $x,\xi\in\re$,
$f_1(x) = 0$,  $y\in\re^2$, 
$\mu\sim\mc{U}(S)$ and
\[
X = \{x\in\re: 1-x^2\ge 0\},\quad S = \{\xi\in\re: \xi(1-\xi)\ge 0\},
\]
where $\mu\in\mc{U}(S)$ denotes the uniform probability measure.
The second-stage problem is given by 
\[
\left\{
\begin{array}{rl}
f_2(x,\xi) = \min\limits_{y\in\re^2} & xy_1+2xy_2\\
\st & y_1-x-\xi\ge 0,\\
& y_2-x+\xi\ge 0,\\
& 2x+3\xi-y_1-y_2\ge 0.
\end{array}
\right.
\]
Clearly, the second-stage optimization problem is feasible for every $x\in X$ and $\xi\in S$,
so we have $\mc{F} = X\times S$.
Select $\alpha = 0.1$, $\epsilon = 0.1$, and let $\nu$ be the uniform probability measure 
supported on $\mc{F}$. Apply Algorithm~\ref{def:alg} to this problem.
We consider two different relaxation orders $
\mbox{(i)}\, \mathbf{k} = (2,4,3)$;
$\mbox{(ii)}\, \mathbf{k} = (4,4,4)$.

\vskip 0.05in
\noindent
(i) When $\mathbf{k} = (2,4,3)$, Algorithm~\ref{def:alg} terminates at the loop $t=4$. 
We record the computed polynomial objective approximations in each loop below:
\[\begin{array}{ll}
\tilde{f}_1(x) = -0.4330+1.0000x+1.7010x^2, &
\tilde{f}_2(x) = -0.2500+1.0000x+0.7498x^2,\\
\tilde{f}_3(x) = -0.4330+1.0000x+1.7009x^2, &
\tilde{f}_4(x) = -0.2586+1.0000x+0.8248x^2.
\end{array}
\]
The computed solutions and lower/upper bounds for the optimal values at each iteration are 
listed in Table~\ref{tab:ex:x1y2_sim}.
\begin{table}[htb!]
\footnotesize
\caption{Computational results with $\mathbf{k} = (2,4,3)$ for Example~\ref{ex:dis_rel}}
\label{tab:ex:x1y2_sim}
\centering
\begin{tabular}{|c|c|c|c|c|}
\hline
$t$ &  $1$ & $2$ & $3$ & $4$ \\ \hline
$\tilde{x}^{(t)}$ & $-0.2939$ & $-0.6668$ & $-0.2939$ & $-0.6062$\\ \hline
$\tilde{f}_t(\tilde{x}^{(t)})$ & $-0.5800$ & $-0.5834$ & $-0.5800$ & $-0.5617$ \\ \hline
$f(\tilde{x}^{(t)})$ & $-0.4756$ & $-0.3331$ & $-0.4756$ & $-0.4131$\\ \hline
$\tt diff$ & $0.1044$ & $0.1044$ & $0.1044$ & $0.0861$\\
\hline
\end{tabular}
\end{table}
To evaluate $f(\tilde{x}^{(t)})$, we solve the second-stage optimization problem  by Moment-SOS 
relaxations and use the sample average of 
$\{f_2(\bullet,0.01\cdot i)\}_{i\in [100]}$.
The output solution and the best lower bound of the optimal value are
\[
\tilde{x}^* \,=\, \tilde{x}^{(1)}\,=\, -0.2939,\quad 
\tilde{f}^* \,=\, \tilde{f}_4(\tilde{x}^{(4)}) \,=\, -0.5617.
\]

\noindent
(ii) When $\mathbf{k} = (4,4,4)$, Algorithm~\ref{def:alg} terminates at the initial loop 
$t = 1$ with the polynomial objective approximation
\[
\tilde{f}_1(x) \,=\, -0.3035+1.0000x+1.0034x^2+0.7999x^4.
\]
By solving optimization problem \reff{eq:alg:lb}, we get the solution  and the lower bound of the optimal value 
\[
\tilde{x}^* = \tilde{x}^{(1)} \,=\,  -0.3979,\quad \tilde{f}^* = \tilde{f}_1(\tilde{x}^{(1)}) \,=\, -0.5225.
\]
We again evaluate $f(\tilde{x}^*)$ by the sample average of $\{f_2(\bullet,0.01\cdot i)\}_{i\in [100]}$ and obtain
\[
f(\tilde{x}^*) = -0.5198,\quad {\tt diff} = f(\tilde{x}^{(1)}) - \tilde{f}_1(\tilde{x}^{(1)}) = 0.0027<0.1.
\]
Compared to the previous case, it is clear that the increase of the relaxation order leads to a better 
polynomial approximation and a smaller gap between the upper and lower bounds of objective values.

An important usage of the above computed lower bounds of the objective value is to certificate the  quality of 
a (local) solution obtained by other methods. To illustrate this, we consider the solutions computed by the decomposition algorithm 
proposed in \cite{LiCui22} to solve the current example,  with the same parameters selected in the
reference. The latter method is only guaranteed to compute a properly defined first-order stationary point and it is likely that the computed objective value is far from globally optimal. We consider $100$ scenarios over $5$ independent replications and select the initial point $x = 0$.
The computational results are reported in the following table.
\begin{table}[htb]
\caption{Computational results with the decomposition algorithm}
\label{tab:dcp_meth}
\centering
\begin{tabular}{|l|c|c|c|c|c|}
\hline
Test number & $1$ & $2$ & $3$ & $4$ & $5$\\ \hline
 Output point & $-0.4200$ & $-0.4438$ & $-0.4091$ & $-0.3926$ & $-0.4147$\\ \hline
 Output value & $-0.5042 $ & $-0.5627$ & $-0.4784$ & $-0.4406$ & $-0.4915$\\ \hline
\end{tabular}
\end{table}

In the above table, 
the output objective value $-0.5042$ from Test 1 is the closest to  our computed 
lower bound  $f(\tilde{x}^*) = -0.5225$. This may suggest that the computed objective value in this test is close 
to the true globally optimal value of the two-stage SP, and the output point $x = -0.4200$ can be viewed as an approximate  global  solution.

In addition, we plot the expected recourse function $f$ (evaluated via sample averages) and computed polynomial 
lower bound functions in Figure~\ref{fig:uni_x1y2}.
The left subfigure is for the case $\mathbf{k} = (2,4,3)$, and the right subfigure is for the
case $\mathbf{k} = (4,4,4)$. In both subfigures, $f$ is plotted with solid lines and
$\tilde{f}_1$ is plotted with dashed lines. In the left panel, $\tilde{f}_2$ is plotted with the dotted line, 
$\tilde{f}_4$ is plotted with the dash-dotted line.
\begin{figure}[htb!]
\centering
\includegraphics[width=0.5\textwidth]{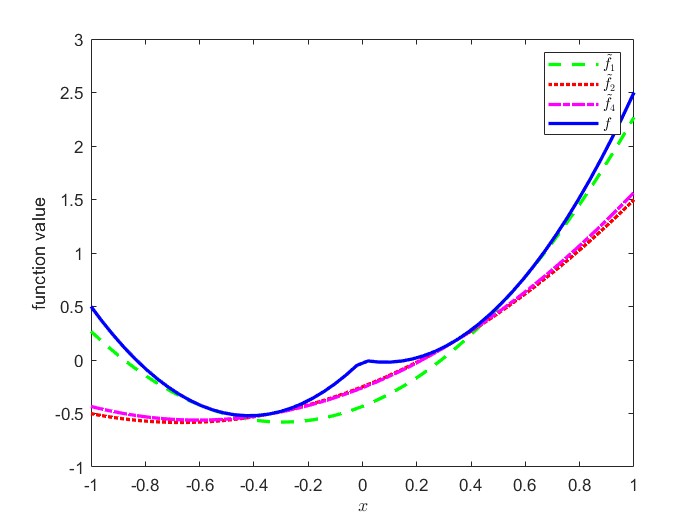}\hfill
\includegraphics[width=0.5\textwidth]{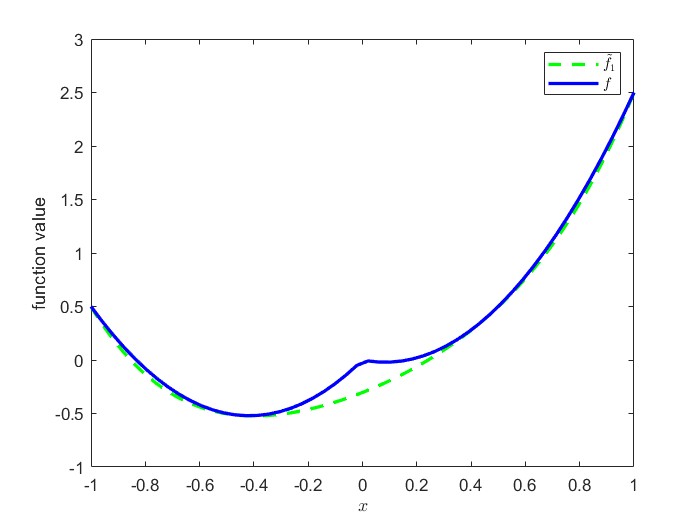}
\caption{The recourse function and its approximations for Example~\ref{ex:x1y2_sim}.
The left panel is for $\mathbf{k} = (2,4,3)$, and the right panel is for $\mathbf{k} = (4,4,4)$.
The  dashed line is for $\tilde{f}_1$, dotted line is for $\tilde{f}_2$, dash-dotted line is for $\tilde{f}_4$ 
and the solid line is for $f$.}
\label{fig:uni_x1y2}
\end{figure}
It can be observed that the polynomial approximation with order $\mathbf{k} = (4,4,4)$ 
also gives a better approximation to the true optimizer compared to the case $\mathbf{k} = (2,4,3)$.
On the other hand, a small increase of the relaxation order can heavily enlarge the dimension of 
the corresponding linear conic optimization problem \reff{eq:polysos:k}.
For the case $\mathbf{k} = (2,4,3)$, there are $210$ scalar variables, $7$ matrix variables equivalent to $1350$ scalar variables when scalarized, and 
$1365$ constraints. 
In contrast,  for the case where $\mathbf{k} = (4,4,4)$, there are $495$ scalar variables, 
$7$ matrix variables which scalarize to $6265$ variables, and $6290$ constraints.
\end{example}

Next we apply 
Algorithm~\ref{def:alg} to a problem of a larger scale, where the second-stage variable $y\in \mathbb{R}^{10}$. 
In this example, the polynomial lower approximating functions again yield a high-quality solution 
with relatively low degrees.
\begin{example}\label{ex:bilinear}
Consider the two-stage SP 
\[
\left\{\begin{array}{cl}
\min\limits_{x\in\re^2} & x_1x_2+\mathbb{E}_{\mu}[f_2(x,\xi)]\\
\st & 1-x_1^2-x_2^2\ge 0,
\end{array}
\right.
\]
where $\xi\in S=[0,1]$ follows a  uniform distribution on $S$, 
and $f_2(x,\xi)$ is the optimal value function of the problem 
(here $e = (1,\ldots, 1)^T\in\re^{n_2}$ is the vector of all ones)
\[
\left\{\begin{array}{cl}
\min\limits_{y\in\re^{10}} & \|y\|^2-y_1^2-\xi\cdot e^Ty\\
\st & (x_2+2)y_1-x_1+2\xi\ge 0,\\
& 2+x_2+(x_1-2)y_2\ge 0,\\
& 10-x_1-e^Ty\ge 0,\\ 
& y_i\ge 0, \,i = 2,\ldots, 9.
\end{array}
\right.
\]
Then $\mc{F} = X\times S$ since the second-stage problem is feasible for every $x\in X$ and $\xi\in S$.
Now we apply Algorithm~\ref{def:alg} to this problem.
We select $\alpha = 0.1$, $\epsilon = 0.06$, and let $\nu$ be the uniform
probability measure supported on $\mc{F}$. 
Denote by $\tilde{p}_t(x)$ the computed lower bound function for 
$f_2(x,\xi)$ at the $t$th loop.
For the degree bound $\mathbf{k} = (2,2,2)$, 
we obtain polynomial objective approximations
\[
\begin{array}{l}
\tilde{f}_1(x) \,=\, x_1x_2+(-5.0868+0.4978x_1
   -0.0023x_2-0.0069x_1^2-0.0113x_2^2),\\[0.05in]
\tilde{f}_2(x) \,=\, x_1x_2+(-5.0894+0.4730x_1
    +0.0310x_2-0.0192x_1^2\\
    \qquad \qquad \qquad\qquad  -0.0434x_1x_2-0.0477x_2^2).
\end{array}
\]
By solving optimization problem \eqref{eq:alg:lb}, 
we get optimal solutions for each approximation and corresponding lower/upper bounds for the optimal value:
\[\begin{array}{lll}
\tilde{x}^{(1)} \,=\,  (-0.8033,0.5956)^T, & \tilde{f}_1(\tilde{x}^{(1)}) \,=\, -5.9750, &
f(\tilde{x}^{(1)}) \,=\, -5.8801,\\
\tilde{x}^{(2)} \,=\,  (-0.8037,0.5950)^T, & \tilde{f}_2(\tilde{x}^{(2)}) \,=\, -5.9379, &
f(\tilde{x}^{(2)}) \,=\, -5.8801.
\end{array}\]
In the above, each $f(\tilde{x}^{(t)})$ is approximated by the sample average of 
$\{f_2(\bullet,0.01\cdot i)\}_{i\in [100]}$.
Since 
\[
{\tt diff} \,=\, f(\tilde{x}^{(2)}) - \tilde{f}_2(\tilde{x}^{(2)}) \,=\, 0.0578 \,<\, 0.06,
\]
we have that Algorithm~\ref{def:alg} terminates at the loop $t = 2$.
\end{example}

Our last test example is 
a joint shipment planning and pricing problem, 
which can be modeled in the form of two-stage SP \cite{LiuCui20}. 
\begin{example}\label{ex:application}
Consider one product in a network consisting of $M$ factories and $N$ retailer stores.
For each $i\in [M]$, factory $i$ has an initial schedule to produce
the product with amount $x_i$ at cost $c_{1i}$ per unit,
and it may allow additional production with amount $y_i$ at cost 
$c_{2i}>c_{1i}$ per unit.
In addition, to ship a unit of item from factory $i$ to store $j$ cost $s_{ij}$.
Let $x_0$ denote the product price and $z_{ij}$ denote the product amount shipped from factory $i$ to store $j$. 
The goal is to fulfill the demand with the lowest cost.
Suppose the demand is linearly dependent on the price $x_0$ and some 
random vectors $\xi = (\xi_1,\ldots,\xi_{n_0})$.
In addition, suppose there exist highest price and production limits. 
That is, there are scalars $d_0, d_{1,i},d_{2,i}>0$ such that $x_0\le d_{0}$
and $x_i\le d_{1,i},\, y_i\le d_{2,i}$ for every $i\in[M]$.
Let 
\[ c_j = (c_{j,1},\ldots, c_{j,M})^T,\quad 
d_j = (d_{j,1},\ldots, d_{j,M})^T,\quad j = 1,2.
\]
The shipment planning problem can be formulated as
\[\left\{\begin{array}{cl}
\min\limits_{x_0\in\re} & \mathbb{E}_{\mu}[f_2(x_0,\xi)]\\
\st & d_0\ge x_0\ge 0,
\end{array}\right.\]
where $f_2(x_0,x,\xi)$ is the optimal value of
\[
\left\{\begin{array}{cl}
\min\limits_{(x,y,z)} & c_1^Tx + c_2^Ty + 
\sum\limits_{i=1}^M\sum\limits_{j=1}^N(s_{ij}-x_0)z_{ij}\\[0.05in]
\st & a_j(\xi)x_0+b_j(\xi)-\sum\limits_{i = 1}^M z_{ij}\ge 0,\,\forall j\in[N],\\[0.05in]
 & x_i+y_j - \sum\limits_{j = 1}^N z_{ij}\ge 0,\, \forall i\in[M],\\[0.05in]
 & d_1\ge x\ge 0,\, x = (x_1,\ldots, x_M)\in\re^M\\[0.05in]
 & d_2\ge y\ge0,\, y = (y_1,\ldots, y_M)\in\re^{M},\\[0.05in]
 & z\ge 0,\, z = (z_{ij})_{i\in[M],j\in[N]}\in\re^{M\times N}.
\end{array}
\right.
\]
Up to a proper scaling, suppose the parameters are selected as
\[
\begin{array}{lllllll}
\hline
M = 2, & N = 3, & d_0 = 1, & d_{1,1} = 1, & d_{1,2} = 1, & d_{2,1} = 1\\
d_{2,2} = 1 & c_{1,1} = 0.2, & c_{1,2} = 0.2, & c_{2,1} = 0.44, & c_{2,2} = 0.46, & s_{1,1} = 0.1,\\
s_{1,2} = 0.2, & s_{1,3} = 0.3,& s_{2,1} = 0.3, & s_{2,2} = 0.2, & s_{2,3} = 0.1.\\
\hline
\end{array}
\]
(i) Consider $\xi = (\xi_1,\xi_2)$ whose probability measure $\mu$ follows the 
truncated standard normal distribution supported on $S = [0,1]^2$.
We set
\[\begin{array}{lll}
\hline
a_1(\xi) \,=\, -2\xi_1, & a_2(\xi) \,=\, -2.5(\xi_1+0.01), & a_3(\xi) \,=\, -3\xi_1-0.06,\\
b_1(\xi) \,=\, 0.5\xi_2+3, & b_2(\xi) \,=\, 0.7\xi_2+4, & b_3(\xi) \,=\, -0.1\xi_2+5.\\
\hline
\end{array}
\]
Now we apply Algorithm~\ref{def:alg} to this problem.
We generate $500$ independent samples following the distribution  $\mu$.
Select $\alpha = 0.1$, $\epsilon = 0.3$ and let $\nu$ be the Cartesian product of $\mu$ 
and the uniform probability measure supported on $X$.
For the relaxation order $\mathbf{k} = (2,2,2)$,
Algorithm~\ref{def:alg} terminates at the loop $t=3$. To improve the approximation, 
we execute  two more iterations, and obtain the following objective approximations:
\[
\begin{array}{l}
\tilde{f}_1(x_0) \,=\, 
-0.05386+0.8499x_0-3.3528x_0^2, \\
\tilde{f}_2(x_0) \,=\, -0.5393x_0-1.7613x_0^2,\\
\tilde{f}_3(x_0) \,=\, -0.5943x_0-1.7057x_0^2, \\
\tilde{f}_4(x_0)\,=\,-0.0019-0.5958x_0-1.7023x_0^2\\
\tilde{f}_5(x_0) \,=\, -0.2300-0.2653x_0-1.8046x_0^2.
\end{array}
\]
We report the computational results in Table~\ref{tab:ex:shipment} and plot the expected 
recourse function and its approximations in the left subgraph of Figure~\ref{fig:shipmement}.
In the figure, $f$ is plotted in the solid line, $\tilde{f}_1$ is plotted in the dashed line, 
$\tilde{f}_2$ is plotted in the dotted line and $\tilde{f}_5$ is plotted in the dash-dotted line.
\begin{table}[htb!]
\caption{Computational results with $\mathbf{k} = (2,2,2)$ for Example~\ref{ex:application}}
\label{tab:ex:shipment}
\centering
\begin{tabular}{|c|c|c|c|c|c|}
\hline
$t$ &  $1$ & $2$ & $3$ & $4$ & $5$\\ \hline
$\tilde{x}_0^{(t)}$ & $1.0000$ & $1.0000$ & $1.0000$ & $1.0000$ & $1.0000$\\ \hline
$\tilde{f}_t(\tilde{x}_0^{(t)})$ & $-2.5568$ & $-2.3005$ & $-2.3000$ & $-2.3000$ & $-2.2999$ \\ \hline
$f(\tilde{x}_0^{(t)})$ & $-2.1000$ & $-2.1000$ & $-2.1000$ & $-2.1000$ & $-2.1000$\\ \hline
${\tt diff}$ & $0.4580$ & $0.2018$ & $0.2012$ & $0.2012$ & $0.2011$\\
\hline
\end{tabular}
\end{table}

\noindent
(ii) Consider the situation that  $a_j(\xi)$ and $b_j(\xi)$ have the following finite realizations with equal probabilities:
\[\begin{array}{lll}
\hline
a_1(\xi)\in\{-0.5,-2\} & a_2(\xi)\in\{-3\} & a_3(\xi)\in\{-1,-3\}\\
b_1(\xi)\in\{3\}, & b_2(\xi)\in\{4,7\}, & b_3(\xi)\in\{5\}\\
\hline
\end{array}
\]
We apply Algorithm~\ref{def:alg:dis} to this problem.
Select $\alpha = 0.1$, $\epsilon = 0.3$ and let each $\nu_i$ be the uniform probability measure supported on $X$. 
For the relaxation order $k = 4$, Algorithm~\ref{def:alg:dis} terminates at the loop $t=2$ with the following objective approximations:
\[
\begin{array}{l}
\tilde{f}_1(x_0) \,=\,-0.2981+2.9749x_0-8.0524x_0^2+3.0694x_0^3,\\
\Tilde{f}_2(x_0) \,=\, -0.3001+2.9921x_0-8.0952x_0^2+3.1002x_0^3.
\end{array}
\]
The output solution and the corresponding bounds for the optimal value are
\[
\tilde{x}_0^* = 1.0000,\quad \tilde{f}^* = -2.3030,\quad f(\Tilde{x}^*) = -2.0500.
\]
The gap ${\tt diff} = f(\tilde{x}^*) - \tilde{f}^* = 0.2530<0.3$.
We plot the expected recourse function and its polynomial approximation in right subgraph 
of Figure~\ref{fig:shipmement}.
In the figure, $f$ is plotted in the solid line, $\tilde{f}_1$ is plotted in the dashed line.
\begin{figure}[htb!]
\centering
\includegraphics[width=0.5\textwidth]{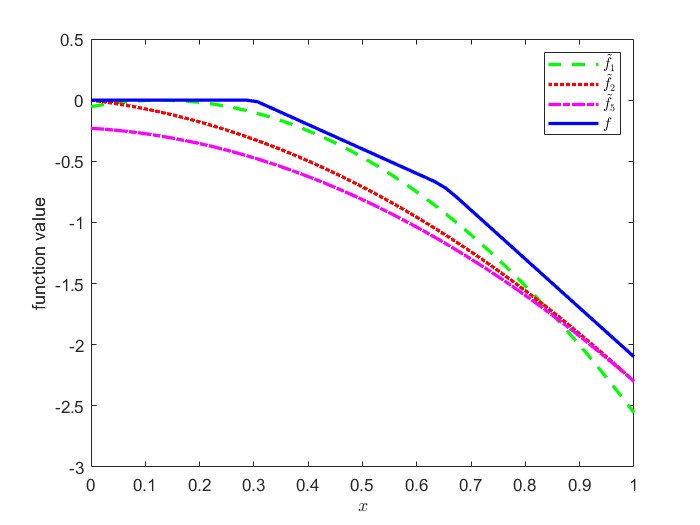}\hfill
\includegraphics[width=0.5\textwidth]{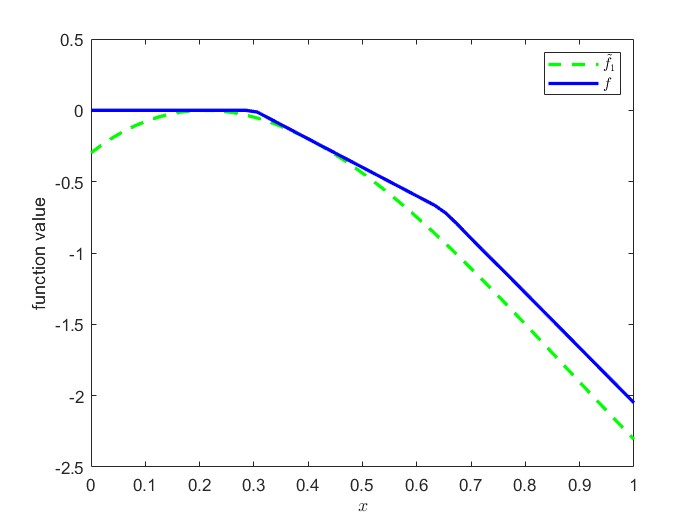}
\caption{The recourse function and its approximations for Example~\ref{ex:application}.
The left is for case (i) and the right is for case (ii).
In both subgraphs, the  dashed line is for $\tilde{f}_1$ and and the solid line is for $f$.
Particularly in the left, dotted line is for $\tilde{f}_2$ and dash-dotted line is for $\tilde{f}_5$. 
}
\label{fig:shipmement}
\end{figure}
It is clear that our polynomial approximating bound functions provide good approximations to the true objective function.
\end{example}

\section{Conclusions}
\label{sec:con}
In this paper, we have explored a novel computational  method for computing global optimal solutions of two-stage stochastic programs through polynomial optimization. Our proposed method hinges on the computation of the polynomial lower bound  of the recourse function.
These lower bound functions can be determined by the solutions of a sequence of linear conic optimization problems, 
where the size of the decision variable does not depend on the number of scenarios in the second stage problem.
The approach presents significant computational advantages.
It can identify a tight lower bound for the global optimal value of \reff{eq:2stageSP}, 
which can be used to certify the global optimality of a candidate solution obtained by other methods.
Furthermore, our method is notably effective when the random variables follow empirical distributions with a large number of scenarios or continuous distributions.
In the future, we plan to further explore the structure of the two-stage stochastic problems so that our proposed  approach can be used to solve large-scale problems more efficiently.
We also aim to improve the efficiency of polynomial lower approximating functions, particularly for those with low degrees.
In addition, we anticipate that our proposed approach can be generalized to cases where the distribution of $\xi$ depends
on $x$. We plan to explore this as future work.

\section*{Acknowledgments}
The project was partially done when the authors attended the SQuaRE program hosted by the American Institute for Mathematics (AIM). The authors thank AIM for providing a supportive and collaborative environment. The authors are also grateful to the associate editor and reviewers for their constructive suggestions, which has helped to improve the manuscript.

\bibliographystyle{siamplain}

\end{document}